\documentclass{amsart}
\usepackage[all]{xy}
\usepackage{amsmath,amsfonts,amssymb,amsthm,hyperref}
\usepackage{mathrsfs}
\usepackage{color}

\allowdisplaybreaks

\usepackage[margin=3.7cm]{geometry}
\usepackage{hyperref}
\usepackage{cleveref}
\usepackage{setspace}
\usepackage{enumitem}
\usepackage{verbatim}

\newtheorem{theorem}{\bf Theorem}[section]
\newtheorem{lemma}[theorem]{\bf Lemma}
\newtheorem{proposition}[theorem]{\bf Proposition}
\newtheorem{definition}[theorem]{\bf Definition}
\newtheorem{corollary}[theorem]{\bf Corollary}
\newtheorem{remark}[theorem]{\bf Remark}
\newtheorem{example}[theorem]{\bf Example}
\numberwithin{equation}{section}

\newcommand{\omin}{\otimes^{\min}}
\newcommand{\la}{\langle}
\newcommand{\ra}{\rangle}
\newcommand{\ot}{\otimes}
\newcommand{\Ind}{\mathrm{Ind}}
\newcommand{\p}{\mathrm{p}}
\newcommand{\cp}{\mathrm{cp}}

\newcommand{\IMS}{\mathrm{IMS}}
\newcommand{\C}{\mathbb{C}}

\newcommand{\mF}{\mathcal{F}}
\newcommand{\mH}{\mathcal{H}}
\newcommand{\mK}{\mathcal{K}}
\newcommand{\mL}{\mathcal{L}}

\newcommand{\mN}{\mathcal{N}}
\newcommand{\mM}{\mathcal{M}}
\newcommand{\mP}{\mathcal{P}}
\newcommand{\mQ}{\mathcal{Q}}

\newcommand{\mX}{\mathcal{X}}
\newcommand{\mY}{\mathcal{Y}}
\newcommand{\mZ}{\mathcal{Z}}

\newcommand{\mA}{\mathcal{A}}
\newcommand{\mB}{\mathcal{B}}
\newcommand{\mC}{\mathcal{C}}
\newcommand{\mD}{\mathcal{D}}

\newcommand{\mcal}{\mathcal}
\begin{document}

\title[Angles between intermediate operator subalgebras]{Angles between
 intermediate operator subalgebras}

\author[V P Gupta]{Ved Prakash Gupta}
\author[D Sharma]{Deepika Sharma}

\address{School of Physical Sciences, Jawaharlal Nehru University, New
  Delhi, INDIA}

\email{vedgupta@jnu.ac.in}
\email{sharmadeepikaq@gmail.com}


\date{}

\begin{abstract}
Motivated by \cite{BDLR} and \cite{BG2}, the notions of interior and
exterior angles between a pair of compatible intermediate
$W^*$-subalgebras of an inclusion of $W^*$-algebras with a normal
conditional expectation with finite probabilistic index are
introduced. This is then employed effectively to define the 
interior angle between a pair of compatible intermediate
$C^*$-subalgebras of an inclusion of non-unital $C^*$-algebras with a
conditional expectation with finite Watatani index. It is also shown
that the interior angle is stable under the minimal tensor product of
unital $C^*$-algebras.
\end{abstract}
\maketitle

\section{Introduction}
Over the years, based on some significant observations in various
categories - like those of groups, von Neumann algebras and
$C^*$-algebras - it is a well utilized strategy in mathematics that
the study of sub-objects of a given mathematical object provides a
good understanding of the structure and various properties of the
ambient object. In fact, based on this methodology, some noteworthy
classification results have been achieved in the categories of groups,
von Neumann algebras and $C^*$-algebras.

There are various tools and techniques in the world of operator
algebras which assist in the study of subalgebras of operator
algebras. Very recently, the following concepts were added to this
list:
\begin{enumerate}
  \item The notions of interior and exterior angles between
    intermediate subfactors of any inclusion of $II_1$-factors with
    finite Jones index, by Bakshi et al in \cite{BDLR}. (The existence
    of a unique trace on a $II_1$-factor and Jones' notion of `basic
    construction' played a fundamental role in the development and
    applications of these concepts.)
\item The notions of interior and exterior angles between intermediate
  $C^*$-subalgebras of any inclusion of unital $C^*$-algebras with
  finite Watatani index, by Bakshi and the first named author in
  \cite{BG2}. (Watatani's theories of finite-index conditional
  expectations and reduced $C^*$-basic construction (based on the
  theory of Hilbert $C^*$-modules) were instrumental for these notions 
  of angles.)
\end{enumerate}

As applications, these notions of angles were used effectively to
obtain bounds for the cardinalities of lattices of intermediate
subalgebras of irreducible finite-index inclusions of (simple)
operator algebras mentioned as above, in \cite{BDLR} and \cite{BG2}
itself. These bounds improved earlier known bounds by Longo
(\cite{Longo}), thereby answering a question asked by him; and, also
provided a bound for the cardinality of intermediate subfactors of a
finite-index irreducible inclusion of type $III$-factors. (Some
further improvements were achieved in \cite{BGJ}.)

It was then natural to ask whether these notions of angles could be
extended to intermediate subalgebras of more general (finite-index)
inclusions of $W^*$-algebras and of non-unital $C^*$-algebras. This
article essentially addresses these questions and, interestingly, it
turns out that the language of Hilbert $W^*$-algebras developed by
Paschke plays a key role in answering both questions effectively.

Here is a brief outline of the flow of this article.

This being our first serious interaction with the theory of Hilbert
$W^*$-modules, we've included a relatively longish portion on facts
and concepts related to Hilbert $W^*$-modules in \Cref{prelims}.
Another major component of \Cref{prelims} involves a brief review of
the notion of conditional expectations with finite Watatani index in
unital as well as non-unital $C^*$-algebras, which was initiated by
Watatani (\cite{watatani}) for unital $C^*$-algebras and later
generalized to non-unital $C^*$-algebras by Izumi (\cite{izumi}). (We
anticipate that the exploratory survey presented in \Cref{prelims} makes
this article more self-contained and that it will benefit some
uninitiated readers.) Then, following \cite{watatani} and
depending heavily on \cite{Pas} and \cite{BDH},
\Cref{W*-basic-construction} reviews some basic aspects of {\em
  reduced $W^*$-basic construction} of an inclusion of von Neumann
algebras with a normal conditional expectation with finite
probabilistic index, wherein the notion of a {\em generalized
  quasi-basis} for a conditional expectation in $W^*$-algebras has
also been introduced (and discussed briefly).

Thereafter, motivated by \cite{BDLR} and \cite{BG2}, the notions of
interior and exterior angles between a pair of compatible intermediate
$W^*$-subalgebras of an inclusion of $W^*$-algebras with a normal
conditional expectation with finite probabilistic index is introduced
in \Cref{W*-angles}. Some expected generalities are also discussed in
the same section. Finally, in \Cref{non-unital-angle}, employing the
notion of the interior angle between intermediate $W^*$-algebras,
defined in \Cref{W*-angles}, and adhering to Izumi's (\cite{izumi})
strategy of working with enveloping $W^*$-algebras of non-unital
$C^*$-algebras, the notion of interior angle between compatible
intermediate $C^*$-subalgebras of an inclusion of (non-unital)
$C^*$-algebras with a conditional expectation with finite Watatani
index is introduced. The article concludes with the observation that
the interior angle is stable under the minimal tensor product of
unital $C^*$-algebras.

\section{Preliminaries}\label{prelims}

 \subsection{Hilbert $C^*$- and $W^*$-modules}\label{module-basics}

 In this section, we just recall a few facts relevant to this
 article. For more on Hilbert $C^*$- and $W^*$-modules, see, for
 instance, \cite{BM, Khos, Lance, Pas, Rieffel}. (All modules
 considered in this article will be right modules.)

Let us first fix notations for the space of intertwiners and for the
space of adjointable maps between pre-Hilbert $C^*$-modules:

      For any two pre-Hilbert $C^*$-modules $\mX$ and $\mY$ over a
      $C^*$-algebra $\mA$, let  $\mB(\mX, \mY)$ denote the space of
      bounded linear maps from $\mX$ into $\mY$, let
\[
\mB_\mA(\mX, \mY):=\{ T \in \mB(\mX,   \mY) :\, T
\text{ is  } \mA \text{-linear}\},
\]
and, let $\mL_\mA(\mX, \mY)$ denote the space of adjointable maps from
$\mX$ into $\mY$, i.e.,
\begin{eqnarray*}
\mL_\mA(\mX, \mY)& := & \{ T : \mX\to \mY : T \text{ is linear and
  there exists a linear map } T^*:\mY \to \mX\\ & & \qquad \text{ such
  that } \la T(x), y\ra_{_\mA} = \la x, T^*(y)\ra_{_\mA} \text{ for
  all } x\in \mX, y \in \mY\}.
\end{eqnarray*}
If $\mX$ and $\mY$ are Hilbert $C^*$-modules over $\mA$, then
$\mL_\mA(\mX, \mY)$ is a closed subspace of the Banach space
$\mB_\mA(\mX, \mY)$ (with the operator norm) - see, for instance,
\cite{Lance, BM}.  It is easily seen that $T^* \in \mL_\mA(\mY, \mX)$,
for every $T \in \mL_\mA(\mX, \mY)$.  Also, for the sake of brevity,
let $\mB_\mA(\mX):= \mB_\mA(\mX, \mX)$ and $\mL_\mA(\mX):=
\mL_\mA(\mX, \mX)$. (Paschke (on page 447 of \cite{Pas}) gave an
example of a Hilbert $C^*$-module $\mX$ over a $C^*$-algebra $\mA$ for
which $\mL_\mA (\mX) \subsetneq \mB_\mA(\mX)$.)

The following fact is elementary - see, for instance, \cite[$\S
  8.1.7$]{BM} - and it provides an appropriate operator algebra, using
which Watatani (in \cite{watatani}) developed the theory of basic
construction for finite-index inclusions of unital $C^*$-algebras. 

\begin{proposition}\label{L-A-X-C-star-algebra}
Let  $\mX$ be a Hilbert $C^*$-module  over a $C^*$-algebra $\mA$. Then,  \(
\mL_{\mA}(\mX)\) is a
unital $C^*$-algebra.
\end{proposition}

\begin{remark}\label{compact-operators}
  For any Hilbert $C^*$-module $\mX$ over a $C^*$-algebra $\mA$, there
  are some useful  so-called `finite-rank' elements in $\mL_\mA(\mX)$ given by
  $\mX$ itself.

  For $x, y \in \mX$, consider the (so-called `rank-one') operator
  $\theta_{x,y}: \mX \to \mX $ given by  \( \theta_{x,y}(z) = x
  \langle y, z \rangle_{_\mA}, z \in \mX \).  Clearly, $\theta_{x,y}
  \in \mL_{\mA}(\mX)$, $\theta_{x,y}\circ \theta_{w,z} = \theta_{x \la
    y, w\ra_{_\mA}, z}$ and $(\theta_{x,y})^* = \theta_{y,x}$ for all
  $x, y, w, z \in \mX$. Further, for any $T \in \mathcal{L}_\mA(\mX)$
  and any pair $x, y \in \mX$, $\theta_{x,y}\circ T = \theta_{x,T^*(y)}$
  and $ T \circ \theta_{x,y} = \theta_{T(x),y}$. Thus,
  $\mathrm{span}\{\theta_{x,y} : x,y \in \mX\}$ is a two-sided ideal in
  $\mL_\mA(\mX)$ and its closure $\mK_{\mA}(\mX):=
  \overline{\mathrm{span}}\{\theta_{x,y}: x, y \in\mX\}$ is called the
  ideal of ``compact operators'' on the Hilbert $C^*$-module $\mX$.
\end{remark}

For any set $I$, let $\mF(I)$ denote the set of finite subsets of
$I$. Clearly, $\mF(I)$ is a directed set with respect to set
inclusion. The following useful observation appears, indirectly, in
\cite[Pages 457-458]{Pas}. We shall need it ahead.

\begin{lemma}\cite{Pas}\label{paschke-convergence}
Let $\mX$ be a pre-Hilbert $C^*$-module over a $W^*$-algebra $\mM$ and
$\{x_i: i \in I\}$ and $\{y_i: i \in I\}$be two family in $\mX$. If
the net $\left\{ \sum_{i \in F} \la x_i, y_i \ra_\mM : F \in\mF(I) \right\}$ is
norm-bounded in $\mM$, then $\sum_i \la x_i, y_i\ra_\mM$ converges in
the $\sigma$-weak topology in $\mM$.
  \end{lemma}

\subsubsection{Self-dual modules} 
Note that, given any pre-Hilbert $C^*$-module $\mX$ over a $C^*$-algebra
$\mA$, for every $z\in \mX$, the map $\hat{z} : \mX \to \mA$ given by
$\hat{z}(x) = \la z, x\ra_\mA$, $x \in \mX$ is a bounded $\mA$-module
map. 
\begin{definition}\cite{Pas}
\begin{enumerate}
    \item A pre-Hilbert $C^*$-module $\mX$ over a $C^*$-algebra $\mA$
      is said to be self-dual if, for every bounded $\mA$-module map
      $T : \mX \to \mA$, there exists a $z \in \mX$ such that $T =
      \hat{z}$.\\  (It turns out that every self-dual pre-Hilbert
      $C^*$-module is a Hilbert $C^*$-module - see (\cite[$\S
        3$]{Pas}).)
      
    \item A  Hilbert $C^*$-module $\mX$ over a $W^*$-algebra $\mM$ is
  said to be a  Hilbert $W^*$-module (over $\mM$) if $\mX$ is  self-dual.
 \end{enumerate}
\end{definition}

\begin{theorem}\cite{Pas, BM}\label{self-dual-facts}
 Let $\mX$ be a Hilbert $W^*$-module over a $W^*$-algebra $\mM$. Then,
 the following hold: \begin{enumerate}

 \item $\mathcal{L}_\mM(\mX) = \mB_\mM(\mX)$. \hfill (\cite[Cor.   3.5]{Pas})
\item $\mX$ is a dual space. \hfill (\cite[Prop. 3.8]{Pas})

\item A bounded net $\{x_\alpha\}$ in $\mX$ converges to some $x \in
     \mX$ in the $w^*$-topology (induced on $\mX$, by the duality
     mentioned in Item (2)) if and only if $\varphi(\la x_\alpha, y
     \ra_{_\mM}) \to \varphi(\la x, y \ra_{_\mM})$ for all $\varphi
     \in \mM_*$ and $y \in \mX$. \hfill (\cite[Rem. 3.9]{Pas})
 \item $\mathcal{L}_\mM(\mX)$ is a $W^*$-algebra.\hfill
   (\cite[Prop. 3.10]{Pas})
 \item A bounded net $\{T_\lambda\}$ in $\mL_\mM(\mX)$ converges
     $\sigma$-weakly to an element $T$ in $\mL_\mM(\mX)$ if and only if
     $T_\lambda(x) \stackrel{w^*}{\rightarrow} T(x)$ in $\mX$ for
     every $x \in \mX$.\hfill (\cite[Cor. 8.5.5]{BM})
\end{enumerate}
\end{theorem}

The following definition is well accepted by the experts - see, for
instance, \cite[Page 357]{BM} or  \cite[Page 78]{Rieffel}.

\begin{definition}
Let $X$ be a pre-Hilbert $C^*$-module over of a $W^*$-algebra $\mM$. A
self-dual completion of $X$ is a self-dual Hilbert $C^*$-module $\mX$
over $\mM$ containing $X$ as a $w^*$-dense pre-Hilbert $C^*$-submodule
(where the $w^*$-topology on $\mX$ comes from the (Paschke's) duality
mentioned in \Cref{self-dual-facts}).
\end{definition}

\begin{theorem}(\cite[Thm. 3.2]{Pas},\cite[Prop. 6.10]{Rieffel})\label{self-dual-completion}
Every pre-Hilbert $C^*$-module over of a $W^*$-algebra $\mM$ admits a
unique self-dual completion.
\end{theorem}

\noindent {\bf Orthogonality in Hilbert $W^*$-modules.}
\begin{definition}\cite[$\S 8.5.21$]{BM} 
  Let $\mX$ be a Hilbert $W^*$-module over a $W^*$-algebra
  $\mM$.\begin{enumerate}
\item
  An element $u \in \mX$ is said to be a partial isometry if $p := \la
  u,u \ra_\mM$ is a projection in $\mM$; and, the projection $p$ is
  called the initial projection of $u$.

\item  Two partial isometries $u$ and $v$
  in $\mX$ are said to be orthogonal if $\la u, v \ra_\mM = 0$.
\end{enumerate}
  \end{definition}

Observe that, for any partial isometry $u$, the rank-one intertwiner
$\theta_{u,u}$ is a projection in the $W^*$-algebra
$\mL_\mM(\mX)$. Also, for two orthogonal partial isometries $u$ and $ v$ in
$\mX$, the rank-one projections $\theta_{u, u}$ and $ \theta_{v,v}$ are
orthogonal in $\mL_\mM(\mX)$.

\begin{definition}(Paschke)\label{W*-onb-defn}
  Let $\mM$ be a $W^*$-algebra and $\mX$ be a Hilbert $W^*$-module
  over $\mM$. An orthonormal basis for $\mX$ is a set $\{x_i: i \in
  I\} \subset \mX$ consisting of mutually orthogonal non-zero partial
  isometries such that $x = \sum_i^{w^*} x_i \la x_i, x\ra_{_\mM}$ for
  all $x \in \mX$, where the $w^*$-topology on $\mX$ comes from the
  (Paschke's) duality mentioned in \Cref{self-dual-facts}(2).
  \end{definition}

\begin{theorem}\cite[Thm. 3.12]{Pas}\label{W*-facts-1}
  Let $\mM$ be a $W^*$-algebra and $\mX$ be a Hilbert $W^*$-module
  over $\mM$. Then, $\mX$ admits an orthonormal basis.
  
{In particular, $\sum_i^{\sigma\text{-weak}} \theta_{x_i, x_i} = \mathrm{id}_\mX$
 in $\mL_\mM(\mX)$}; hence, the two-sided ideal
$\mathrm{span}\{\theta_{x,y}: x, y \in \mX\}$ is $\sigma$-weakly dense in
$\mathcal{L}_\mM(\mX)$.
\end{theorem}
(A more explicit proof of the preceding fact can be found in
\cite[Lemma 8.5.23]{BM}.)

\subsection{Conditional expectations with finite probabilistic index and finite cp-index}

Recall (from \cite{BDH, FK, izumi}) the following well-studied notions of
index of a conditional expectation.
  \begin{definition}
  Let $\mB \subset \mA$ be an inclusion of $C^*$-algebras and $E: \mA
  \to \mB$ be a conditional expectation.
   \begin{enumerate}
\item    
  $E$ is said to have finite probabilistic index if there exists a
  scalar $\gamma \geq 1$ such that $\gamma \, E - \mathrm{id}_\mA$
  \text{ is positive}; and, the probabilistic index of such an $E$ is
  defined as
  \[
  \mathrm{Ind}_\p(E)=\inf \{\gamma \geq 1 : \gamma \, E -
  \mathrm{id}_\mA \text{ is positive} \}.
  \]
\item $E$ is said to have finite $\cp$-index if there exists a scalar
  $\gamma \geq 1$ such that $ \gamma \, E - \mathrm{id}_\mA $ is
  completely positive; and, the $\cp$-index of such an $E$ is defined
  as
        \[
  \mathrm{Ind}_\cp(E)=\inf \{\gamma \geq 1 : \gamma \, E -
  \mathrm{id}_\mA \text{ is completely positive} \}.
  \]
   \end{enumerate}
  \end{definition}
  
For convenience, the notation $\Ind_\p(E)< \infty$ (resp.,
$\Ind_\cp(E)< \infty$) shall mean that the conditional
expectation $E$ has finite probabilistic index (resp., finite
cp-index). We list some useful facts related to these indices.
  \begin{remark}\label{ce-facts}
    Let $\mB \subset \mA$ be an inclusion of $C^*$-algebras and $E: \mA \to
    \mB $ be a conditional expectation. 
    \begin{enumerate}
\item If $\Ind_\p(E)< \infty$, then $E$ is faithful.

\item $\Ind_\p(E) < \infty$ if and only if $\Ind_{\cp}(E) <
   \infty$. (\cite[Thm. 1]{FK})

\item If $\Ind_\cp(E)< \infty$, then $1 \leq \Ind_\p(E) \leq \Ind_\cp(E)$.

 \item If $\mB$ has no non-zero finite-dimensional representations,
  then $\Ind_\p (E) = \Ind_\cp(E)$. (\cite[Lemma 2.2]{izumi})
    \end{enumerate}
For more on the theory of conditional expectations with finite
probabilistic index and finite cp-index, see \cite{ BDH, FK, izumi}.
  \end{remark}
  \begin{remark}\label{Indp-finite-is-intrinsic}
  Interestingly, in some cases, it turns out that the finiteness of
  the probabilistic index of a normal conditional expectation is not a
  property exclusive to that particular conditional expectation. When
  $\mN$ is a factor, it is an intrinsic property of the inclusion $\mN
  \subset \mM$. More precisely, either every normal conditional
  expectation from $\mM$ onto $\mN$ has finite probabilistic index
  (equivalently, finite cp-index), or none has - see
  \cite[Cor. 3.20]{BDH}.
\end{remark}

  \subsection{Conditional expectations with finite Watatani index}

  The third notion of index, which will be central to the theme of
  this article, is due to Watatani (\cite{watatani}), which was
  initially defined for conditional expectations in unital
  $C^*$-algebras and was later generalized by Izumi (\cite{izumi}) to
  general $C^*$-algebras. We briefly recall this notion, for which we
  first need to recall Watatani's theory of $C^*$-basic construction
  and Haagerup's theory of operator valued weights.

  \subsubsection{Watatani's theory of  basic construction for inclusions
    of $C^*$-algebras}\label{basic-construction} Let $\mB \subset \mA$
  be an inclusion of $C^*$-algebras and $E: \mA \to \mB$ be a faithful
  conditional expectation. Then, $\mA$ becomes a (right) pre-Hilbert
  $C^*$-module over $\mB$ with respect to the natural right action
  (via multiplication from the right) and the $\mB$-valued inner
  product $\la \cdot, \cdot \ra_\mB$ given by
  \[
\la x, y \ra_\mB:=E(x^*y), x, y \in \mA.
\]
Let $\mathfrak{A}$ denote the Hilbert $C^*$-module completion (over $\mB$) of
$\mA$, $\eta: \mA \to \mathfrak{A}$ denote the natural inclusion map
(so that, $\langle \eta(x), \eta(y)\rangle_\mB = \eta(E(x^*y))$ and
$\|\eta(x)\|:= \|E(x^*x)^{1/2}\|$ for all $x,y \in \mA$) and
$\mcal{L}_\mB(\mathfrak{A})$ denote the space of adjointable maps on
$\mathfrak{A}$.

As noted in \Cref{module-basics}, $\mL_\mB(\mathfrak{A})$ is a closed
subspace of the Banach space $\mB_\mB(\mathfrak{A})$ consisting of the
$\mB$-module maps on $\mathfrak{A}$. In fact, in view of
\Cref{L-A-X-C-star-algebra}, $\mL_\mB(\mathfrak{A})$ turns out to be a
unital $C^*$-algebra (with the usual operator norm); and, it is easily
seen that there exists a natural $C^*$-embedding $\lambda: \mA \to
\mcal{L}_\mB(\mathfrak{A})$ satisfying $\lambda(a)(\eta(x)) =
\eta(ax)$ for all $a, x \in \mA$.  Further, there exists a projection
$e_E \in \mL_\mB(\mathfrak{A})$ (called the Jones projection
associated to $E$) satisfying $e_E(\eta(a)) = \eta(E(a))$ and \( e_E
\lambda(a) e_E = \lambda(E(a)) e_E \text{ for all } a \in \mA \) - see
\cite[Lemma 2.1.1(1)]{watatani}. (Note that, in literature, $e_E$ is
denoted more often by $e_\mB$ or by $e_1$.) Thus, the closed subspace
\[ 
C^*_r\langle \mA, e_E\rangle :=\overline{\mathrm{span}}\{\lambda(x)
e_E \lambda(y): x, y \in \mA\}
\]
turns out to be a $C^*$-subalgebra of $\mcal{L}_\mB(\mathfrak{A})$,
and is called the (Watatani's reduced) $C^*$-basic construction of the
inclusion $\mB \subset \mA$ with respect to the conditional
expectation $E$. Notice that $C^*_r\langle \mA, e_E\rangle =
\mK_\mB(\mathfrak{A})$, the ideal of compact operators on
$\mathfrak{A}$ (see \Cref{compact-operators}), because $\lambda(x) e_E
\lambda(y^*) = \theta_{\eta(x), \eta(y)}$ for all $x, y \in \mA$. A
priori, $e_E$ need not belong to $C^*_r\langle \mA, e_E\rangle$;
however, when $\mA$ is unital, then $e_E = \theta_{\eta(1), \eta(1)}
\in C^*_r\langle \mA, e_E\rangle$. (Note: If multiple conditional
expectations appear simultaneously, then $\eta$ (resp., $\lambda$)
will be denoted by $\eta_E$ (resp., $\lambda_E$), in order to avoid
any possible confusion.) 

\noindent{\em Some generalities for inclusions of non-unital $C^*$-algebras.}

Generalizing Pimsner-Popa's notion of basis for a conditional
expectation, Watatani (\cite{watatani}) initiated an ``algebraic'' notion of
finite-index conditional expectation in unital $C^*$-algebras (see
\Cref{quasi-basis}). Later, Izumi pointed out in
\cite{izumi} that the right approach to generalize Watatani's notion
of ``finite-index'' conditional expectation to non-unital
$C^*$-algebras is by appealing to their enveloping von Neumann
algebras and not to their multiplier algebras. Here is a brief outline
of his approach.
  
Consider an inclusion $\mB \subset \mA$ of $C^*$-algebras with a
conditional expectation $E: \mA \to \mB$. Then, it is well-known (see
\cite[Cor. 3.7.9]{Ped}) that the enveloping von Neumann algebra
$\mB^{**}$ can be realized as a $w^*$-closed $*$-subalgebra of
$\mA^{**}$ (however, $1_{\mA^{**}}$ need not belong to $\mB^{**}$)
and, with respect to this inclusion, $E$ extends to a normal
conditional expectation $E^{**}: \mA^{**} \to \mB^{**}$ - see
\cite[III.5.2.10]{blackadar}. Also, $E^{**}(1_{\mA^{**}}) =
1_{\mB^{**}}$ (for instance, by a Proposition on Page 118 of \cite{STR});
and, $E^{**}$ is faithful if and only if $E$ is so.

Following \cite{izumi, Ped}, it will be useful to
realize the multiplier algebras $M(\mA)$ and $M(\mB)$ as unital
$C^*$-subalgebras of $\mA^{**}$ and $\mB^{**}$, respectively.

\begin{theorem}\label{crucial-facts}\label{E-to-E**}\cite{izumi}
Let $\mB \subset \mA$ be an inclusion of $C^*$-algebras and $E: \mA
\to \mB$ be a conditional expectation with $\Ind_\p(E) <
\infty$. Then, the following hold:
    \begin{enumerate}
      \item  
        $\Ind_\p(E^{**}) = \Ind_\p(E)$ and $\Ind_\cp(E^{**}) =
        \Ind_\cp(E)$.
\item $M(\mB)$ can be identified with a unital $C^*$-subalgebra of
  $M(\mA)$. In particular, $\mB^{**} \subset \mA^{**}$ is a unital
  inclusion. \hfill (\cite[Lemma 2.6(1)]{izumi})
\item The restriction of $E^{**}$ to $M(\mA)$ is a conditional
  expectation from $M(\mA)$ onto $M(\mB)$.\\ (\cite[Lemma
    2.6(2)]{izumi})
    \end{enumerate}
    \end{theorem}
 
The following elementary observations will be essential and their
analogues for inclusions of unital $C^*$-algebras with respect to
conditional expectations with (finite) quasi-bases (see
\Cref{quasi-basis}) were proved in \cite{watatani}.

\begin{proposition}\label{bc-facts}
Let $\mB\subset \mA$ be an inclusion of $C^*$-algebras and $E:\mA \to \mB$ be
a conditional expectation with $\mathrm{Ind_p(E)} < \infty$. Then, the
following hold:
\begin{enumerate}
\item $\eta: \mA \to \mathfrak{A}$ is a linear contraction.

\item $\mA$ itself is a Hilbert $C^*$-module over $\mB$, i.e.,
  $\mathfrak{A}=\eta(\mA)$.

  \item $\mA = \mA \mB$. In particular, every approximate identity for
    $\mB$ is an approximate identity for $\mA$ as well. \hfill
    (\cite[Lemma 2.6]{izumi})

  \item $\{e_E\}'\cap \lambda(\mA) = \lambda(\mB)$ (in $\mL_\mB(\mA)$). 

  \item $\lambda(A)e_E \cup e_E\lambda(\mA) \subseteq C_r^*\la \mA,e_E \ra$.

\item The map $\psi :\mB\to C_r^*\la \mA,e_E\ra $ given by
$	\psi(b)= \lambda(b)e_E $, $b \in \mB$,
is an injective $*$-homomorphism.
\end{enumerate}
\end{proposition}

\begin{proof}
  (1): Consider the normal conditional expectation $E^{**}: \mA^{**}
  \to \mB^{**}$ as in the paragraph preceding
  \Cref{crucial-facts}. Then, for every $a \in \mA$,
\[
\Vert\eta(a)\Vert^2  :=  \|E(a^*a)\|
          	=  \|E^{**}(a^*a)\| 
          	 \leq  \|E^{**}(\Vert a \Vert^2 \cdot 1_{\mA^{**}})\|
          	 = \Vert a \Vert^2 .
                 \]         

                 (2) follows from $(1)$ on the lines of
                 \cite[Prop. 2.1.5]{watatani}. \smallskip
  
                 (3) follows from  \cite[Lemma 2.6]{izumi}. \smallskip
                 
(4): It follows readily from the definitions that $e_E\lambda(b) =
\lambda(b)e_E$ for all $b \in \mB$. Thus, $\lambda(\mB) \subseteq
\{e_E\}'\cap \lambda(\mA)$. \smallskip

Conversely, suppose that $e_E\lambda(a) = \lambda(a)e_E$ for some $a
\in \mA$.  $\mB$. In view of Item (3), fix an approximate unit
$\{u_\gamma\}$ for $\mA$ which is contained in $\mB$.

By Item (1), $\eta$ is
continuous; so,
\[ 	\eta(a)= 
\eta(\lim_\gamma au_{\gamma }) =   \lim_\gamma \lambda(a)e_E (\eta
        (u_{\gamma})) =  \lim_\gamma e_E \lambda(a) (\eta
        (u_{\gamma})) = \eta (E(\lim_\gamma au_{\gamma})) =
        \eta(E(a)).\]
Since $E$ is faithful, $\eta$ is injective; so, $E(a)=a$, i.e., $a\in
\mB$.
\medskip

(5): It suffices to show that $\lambda(\mA) e_E \subseteq C^*_r\la \mA
, e_E\ra$. As above, fix an approximate unit $\{u_\gamma\}$ for $\mA$
which is contained in $\mB$.  Then, for every $a \in \mA$, we have
\[
\lambda(a) e_E = \lim_\gamma \lambda(au_\gamma) e_E = \lim_\gamma \lambda(a) e_E
\lambda(u_\gamma) \in \overline{\lambda(\mA) e_E \lambda(\mB)} \subseteq C^*_r\la \mA , e_E\ra.
\]

(6): That $\psi$ is a $*$-homomorphism follows from the fact that $
\lambda(\mB) = \{e_E\}'\cap \lambda(\mA)$ - as shown in Item (4).

  To show the injectivity of $\psi$, fix an approximate identity
  $\{u_\gamma\}$ for $\mB$.  If $\psi(b)= 0$, then $\lambda(b)e_E =
  0$; and, since $\eta$ is continuous, we get
  \[
  \eta(b) = \lim_\gamma \eta(bu_\gamma) = \lim_\gamma
  \lambda(b)e_E\eta(u_\gamma) = 0.
  \]
Thus, $\eta$ being injective,  $b= 0$. 
\end{proof}

\subsubsection{Quasi-bases, operator valued weights and  conditional
  expectations with finite Watatani index}

    \begin{definition}\cite{watatani}\label{quasi-basis}
      Let $\mB\subset \mA$ be an inclusion of $C^*$-algebras and $E:
      \mA \to \mB$ be a conditional expectation. A finite set $\{u_i :
      1 \leq i \leq n\}$ in $\mA$ is said to be a (right) quasi-basis
      for $E$ if \( x = \sum_i u_i E(u_i^*x)\ (\text{equivalently, $x
        = \sum_i E(xu_i)u_i^* $ }) \text{ for all } x \in \mA.  \)
    \end{definition}

    \begin{remark}\label{IndW-finite-is-intrinsic}
      \begin{enumerate}
      \item
 Like the property of having a conditional expectation
with finite probabilistic index (see \Cref{Indp-finite-is-intrinsic}),
Watatani observed that the property of having a finite quasi-basis is
also an intrinsic property of the inclusion $\mB \subset \mA$ (of
unital $C^*$-algebras). More precisely, when $\mZ(\mA)$ (or $\mZ(\mB)$
or, equivalently, $\mC_\mA(\mB)$) is finite-dimensional, then either
no conditional expectation from $\mA $ onto $\mB$ has a finite
quasi-basis or every conditional expectation from $\mA $ onto $\mB$ has
a finite quasi-basis - see \cite[Prop. 2.10.2]{watatani}.

\item If $\mB \subset \mA$ is an inclusion of $C^*$-algebras with
  $\mA$ unital and $E: \mA \to \mB$ is a conditional expectation with
  finite quasi-basis, then $\mB$ is unital, $\Ind_\p(E)< \infty$ and
  $E$ is faithful - see \cite{watatani}.
    \end{enumerate} \end{remark}

    The following well-known observations will be needed ahead:
\begin{proposition}\cite{watatani, izumi} \label{quasi-basis-facts}
      Let $\mB \subset \mA$ be an inclusion of $C^*$-algebras, $E: \mA
      \to \mB$ be a conditional expectation and $\{u_i: 1 \leq i \leq
      n\}$ be a finite set in $ \mA$.
      \begin{enumerate}
      \item If $\{u_i: 1 \leq i \leq n\}$ is a quasi-basis for $E$,
        then it is a quasi-basis for the conditional expectation
        $E^{**}: \mA^{**} \to \mB^{**}$ as well.

        In particular, $\Ind_\p (E) = \Ind_\p(E^{**})< \infty$ (and
        $E$ is faithful).
      \item If $E$ is faithful, then $\{u_i: 1 \leq i \leq n\}$ is a
        quasi-basis for $E$ if and only if
          \[
          \sum_{i=1}^n \lambda(u_i) e_E \lambda(u_i)^* =
          \mathrm{id}_{\mathfrak{A}}.
          \]
      \end{enumerate}
      In particular, if $E$ has a finite quasi-basis, then $E$ is faithful and 
      \[
      C^*_r\la \mA, e_E\ra
 = \mathrm{span}\{\lambda(x)e_E\lambda(y): x,y\in \mA\} =
 \mL_\mB(\mathfrak{A}) .
 \]
    \end{proposition}

    \begin{proposition}\label{watatani-index-characterization}
      \cite[Prop. 2.1.5 $\&$ Lemma 2.1.6]{watatani} Let $\mB
      \subset \mA$ be an inclusion of  $C^*$-algebras and $E:
      \mA \to \mB$ be a conditional expectation. Consider  the following
      statements:
\begin{enumerate}
\item $E$ has a  quasi-basis.
\item $\Ind_\p(E) < \infty$ and $\mathrm{id}_{\mathfrak{A}} \in
  C^*_r\la \mA, e_E\ra$.
\item $ \Ind_\p(E) < \infty$ and $\lambda(\mA) \subseteq C^*_r\la \mA,
  e_E\ra$.  \end{enumerate} Then, (1) and (2) are equivalent and each
of them implies (3). In addition, if $\mA$ (and hence $\mB$) is
unital, then all three statements are equivalent.
  \end{proposition}

 \begin{remark}\label{quasi-basis-forces-unitality}
      Watatani called a conditional expectation with a finite quasi-basis to
      be of index-finite type. It turns out that if a conditional
      expectation $E: \mA \to \mB$ has a finite quasi-basis, then $\mA$ and
      $\mB$ must be unital with common identity - see
      \cite[Rem. 2.4 (5)]{izumi} (or, \Cref{dual-ce-unital} ahead).
    \end{remark}

The equivalent statements of \Cref{watatani-index-characterization}
are the natural precursors to the generalization of Watatani's notion
of finite-index index conditional expectations in non-unital
$C^*$-algebras, which was developed by Izumi (in \cite{izumi}) using
Haagerup's theory of operator valued weights. (This requires some
theoretical background from various references, which we include
concisely for the sake of completeness.)\smallskip

\noindent {\em Operator valued weights:}

Recall (from \cite{Haag, STR})
that, for any $W^*$-algebra $\mM$, $\widehat{\mM}_+$ denotes the set
of lower semi-continuous maps $m: \mM^+_* \to [0, \infty]$ satisfying
\[
m(\lambda \varphi) = \lambda m(\varphi) \text{ and } m(\varphi + \psi)
= m(\varphi) + m(\psi)
\]
for all $\varphi, \psi \in \mM^+_*$ and $\lambda \geq 0$, where
$\mM^+_*$ denotes the set of $w^*$-continuous positive linear
functionals on $\mM$. Clearly, $\mM_+$ can be considered as a subset
of $\widehat{\mM}_+$ (via evaluation). 

For any inclusion $\mN \subset \mM$ of $W^*$-algebras, an {\em
  operator valued weight} from $\mM$ to $\mN$ is a map $T: \mM_+ \to
\widehat{\mN}_+$ satisfying the following conditions:
\begin{enumerate}
\item $T(c x ) = c T(x)$;
\item  $T(x + y) = T(x) + T(y)$; and,
\item $T(a^* x a) = a^* T(x) a$,\end{enumerate} for all $x, y \in
\mM_+$, $a \in \mN$ and scalars $c \geq 0$, where $(a^* T(x) a) (\varphi)
:= T(x) ( \varphi(a^*\cdot a))$ for $\varphi \in \mN^+_*$.

An operator valued weight $T$ is said to be  bounded if $T(\mM_+)
\subset \mN_+$; and,  normal if for any increasing limit $x_i \nearrow x$
in $\mM_+$, $T(x_i) \nearrow T(x)$ (pointwise) in $\widehat{\mN}_+$.

\begin{remark}\label{finite-weight-to-ce}
  The restriction of any (normal) conditional expectation $E: \mM \to
  \mN$ to $\mM_+$ is a (normal) bounded operator valued weight that maps
  $1_\mM$ to $1_\mN$; and, conversely, every (normal) bounded operator
  valued weight $T: \mM_+ \to \widehat{\mN}_+$ that maps $1_\mM$ to
  $1_\mN$ extends naturally to a (normal) conditional expectation from
  $\mM$ onto $\mN$ - see \cite{Haag, STR}.
\end{remark}

\begin{theorem}\label{crucial-facts-2}\cite{BDH}
Let $\mB \subset \mA$ be an inclusion of $C^*$-algebras and $E: \mA
\to \mB$ be a conditional expectation with $\Ind_\p(E) <
\infty$. Then, the following hold:
  \begin{enumerate}
  \item $\mA^{**}$ is a Hilbert $W^*$-module over $\mB^{**}$ with
    respect to the canonical $\mB^{**}$-valued inner product induced
    by the faithful normal conditional expectation $E^{**}:\mA^{**} \to
    \mB^{**}$.
    \item Identifying $\mA^{**}$ with $\lambda_{E^{**}}(\mA^{**})$ in
      $\mL_{\mB^{**}}(\mA^{**})$, $W^* \Big( \mA^{**} \cup\{
      e_{E^{**}}\}\Big) = \mL_{\mB^{**}}(\mA^{**}) $, where $
      e_{E^{**}}$ denotes the Jones projection associated to $E^{**}$.
      \hfill (\cite[Prop. 3.3]{BDH})
  \item There exists a bounded normal operator valued weight
    $\widehat{E^{**}}$ from $\mL_{\mB^{**}}( \mA^{**})$ to $\mA^{**}$
    such that
    \[
    \widehat{E^{**}}\big(x e_{\widehat{E^{**}}}x^*\big)
    = xx^* \text{ for all } x \in \mA^{**},
    \]
    and $\widehat{E^{**}}(\mathrm{id}_{\mA^{**}}) \in
    \mZ(\mA^{**})$.\hfill (\cite[Thm. 3.5]{BDH}) \end{enumerate}
    \end{theorem}

In general, a conditional expectation $E: \mA \to \mB$ with finite
probabilistic index need not have finite Watatani index - see, for
instance, \cite[Pages 89-90]{FK}. The following insightful observation
by Izumi paved the way for generalizing Watatani's notion of
index-finite type conditional expectations for inclusions of
non-unital $C^*$-algebras.
\begin{theorem}\cite[Thm. 2.8]{izumi}\label{izumi-characterization}
Let $\mB, \mA$ and $E$ be as in \Cref{crucial-facts-2}. Then,
identifying $\mA$ (resp., $\mA^{**}$) with $\lambda_E(\mA)$ (resp.,
$\lambda_{E^{**}}(\mA^{**})$, the following statements are equivalent:
\begin{enumerate}
\item  $A \subseteq C^*_r\langle \mA,
  e_E \rangle$.
\item  $\widehat{E^{**}}(\mathrm{id}_{\mA^{**}}) \in \mZ(M(\mA))$.
\end{enumerate}
  \end{theorem}

\begin{definition}\cite{izumi}\label{izumi-defn}
Let $\mB \subset \mA$ be an inclusion of $C^*$-algebras. A conditional
expectation $E: \mA \to \mB$ is said to have finite Watatani index if
$\Ind_\p(E) < \infty$ and if it satisfies any of the equivalent
statements of \Cref{izumi-characterization}. Further, the Watatani index of
such a conditional expectation is defined as
\[
\Ind_W(E) = \widehat{E^{**}}(\mathrm{id}_{\mA^{**}}) \in \mZ(M(\mA)).
\]
\end{definition}

The preceding definitions by Izumi indeed generalize Watatani's
definitions:
\begin{proposition}
  \cite{ BDH, izumi, watatani}\label{original-W-index} \label{E-to-E**-unital}
    With notation as in \Cref{crucial-facts-2}, if both $\mB$
    and $\mA$ are unital, then the following hold:
  \begin{enumerate}
    \item $E$ has finite Watatani index if and only if it has a
     finite quasi-basis in $\mA$.
\item If $\{u_i: 1 \leq i \leq n\}\subset \mA$ is a finite quasi-basis
 for $E$, then
     \[
  \mathrm{Ind}_W(E)= \sum_{i=1}^n u_i u_i^* ,
  \]
  which is a positive invertible element in $\mathcal{Z}(\mA)$ and is
  independent of the quasi-basis. (In fact, this is how Watatani
  (in \cite{watatani}) had originally defined the index of a conditional
  expectation of index-finite type.)

 \item If $E$ has finite Watatani index, then so does $E^{**}$ and,
   moreover, $\Ind_W(E^{**}) = \Ind_W(E)$. (In particular,
   $1_{\mB^{**}} = 1_{\mA^{**}}$.)
\end{enumerate}\end{proposition}

\begin{remark} \cite{izumi}
  Let $\mB \subset \mA$ be an inclusion of non-unital $C^*$-algebras
  and $E : \mA \to \mB$ be a conditional expectation with finite
  Watatani index. Then, $E^{**}: \mA^{**} \to \mB^{**}$ need not have
  finite Watatani index.

  Indeed, in \cite[Example 2.15]{izumi}, Izumi gave an example of an
  inclusion $\mB \subset \mA$ of non-unital $C^*$-algebras with a
  conditional expectation $E: \mA \to \mB$ with $\Ind_W(E) = 2$ and
  showed that its extension $E^{**}: \mA^{**} \to \mB^{**}$ does not
  admit a quasi-basis.   In
  particular, in view of \Cref{original-W-index}, $E^{**}$ does not
  have finite Watatani index.
  \end{remark}

In general, a conditional expectation with finite probabilistic index
need not have finite Watatani index. However, in some specific
set up, Izumi made the following remarkable observations:
\begin{theorem}\cite{izumi}
  Let $\mB \subset \mA$ be an inclusion of $C^*$-algebras and $E: \mA
  \to \mB$ be a conditional expectation with $\Ind_\p(E) <
  \infty$.  \begin{enumerate}
    \item If $\mA$ is simple, then $\Ind_W(E)< \infty$ and it is a scalar.
      \hfill
      (\cite[Thm. 3.2]{izumi})
\item If $\mB$ is simple, then $\Ind_W(E) < \infty$.   \hfill
  (\cite[Cor. 3.4]{izumi})
  \end{enumerate}
 Moreover, if $\mA$ is unital, then in both cases $E$ has a
 finite quasi-basis in $\mA$.
\end{theorem}
\begin{example}\cite{izumi}  \begin{enumerate}
    \item Let $\mB \subset \mA$ be an inclusion of unital
      $C^*$-algebras with a conditional expectation $ E : \mA \to \mB$
      of finite Watatani index. Consider the canonical inclusion $\mB
      \omin \mK(\ell^2) \subset \mA \omin \mK(\ell^2)$ of non-unital
      $C^*$-algebras, where $\mK(\ell^2)$ denotes the non-unital
      $C^*$-algebra of compact operators on $\ell^2$. Then, the
      conditional expectation $E \omin \mathrm{id}_{\mK(\ell^2)}:\mA
      \omin \mK(\ell^2) \to \mB \omin \mK(\ell^2)$ has finite Watatani
      index and $\Ind_W (E \omin \mathrm{id}_{\mK(\ell^2)}) = \Ind_W
      (E) \ot 1$. \hfill (\cite[Lemma 2.11]{izumi})
      \item If a finite group $G$ acts on a simple $C^*$-algebra
        $\mA$, then the canonical (averaging) conditional expectation
        $E: \mA \to \mA^G$ has finite Watatani index. Moreover, if the
        action is outer, then $\Ind_W(E) = |G|$. \hfill
        (\cite[Cor. 3.12]{izumi})
\end{enumerate}  \end{example}

\subsubsection{Dual conditional expectation}\label{dual-ce-izumi-section}

Recall from \cite{watatani} that, if $\mB \subset \mA$ is an inclusion
of unital $C^*$-algebras, then for any conditional expectation $E: \mA
\to \mB$ with a quasi-basis, then there exists a dual conditional
expectation $E_1: C^*_r\la \mA, e_E \ra \to \mA$ with a quasi-basis
(see \Cref{dual-ce-unital} for the statement).

More generally, for any conditional expectation $E: \mA \to \mB$ with
$\Ind_\p(E)< \infty$, the bounded normal operator valued weight
$\widehat{E^{**}}$ (as in \Cref{crucial-facts-2}) yields a normal
conditional expectation $ \Ind_W(E)^{-1} \widehat{E^{**}}:
\mL_{\mB^{**}}(\mA^{**}) \to \mA^{**}$.  Then, via some very astute
application of the notions of induced representations and of 
quasi-equivalence between two representations, Izumi (in \cite[$\S
  2$]{izumi}) identified the reduced $C^*$-basic construction
$C^*_r\la \mA, e_E \ra$ with a $\sigma$-weakly dense $C^*$-subalgebra
of the $W^*$-algebra $\mL_{\mB^{**}}(\mA^{**})$; and, when
$\Ind_W(E)< \infty$, it turns out that the map $E_1 := \Ind_W(E)^{-1}
\Big( \widehat{E^{**}}\Big)_{\restriction_{C^*_r\la \mA, e_E\ra}}$ is
a conditional expectation from $C^*_r\la \mA, e_E\ra$ onto $\mA$,
which is called the dual conditional expectation of $E$. 

\begin{theorem}\cite[Thm. 2.8]{izumi}\label{dual-ce-izumi}
Let $\mB \subset \mA$ be an inclusion of $C^*$-algebras and $E: \mA
\to \mB$ be a conditional expectation with finite Watatani
index. Then, identifying $\mA$ with $\lambda_E(\mA)$, there exists a
conditional expectation $E_1: C^*_r\la \mA, e_E \ra \to \mA$ with finite
Watatani index  satisfying
\[
 E_1(x e_E y) = \Ind_W(E)^{-1} xy \text{ for all } x, y \in \mA.
\]
\end{theorem}
We can thus deduce the following well-known fact (see
\Cref{quasi-basis-forces-unitality}).
\begin{corollary}
  \label{dual-ce-unital} 
Let $\mB \subset \mA$ be an inclusion of $C^*$-algebras and $E: \mA
\to \mB$ be a conditional expectation.  If $E$ has a
finite quasi-basis, then there exists a dual conditional expectation $E_1:
C^*_r\langle \mA, e_E\rangle \to \mA$ with a finite quasi-basis.

  In particular, $\mA$ and $\mB$ are both
  unital with common identity.
    \end{corollary}

\begin{proof}
  Since $E$ has a quasi-basis, it follows from
  \Cref{watatani-index-characterization} that it has {finite
    Watatani index} and that $\mathrm{id}_{\mathfrak{A}} \in
  C^*_r\langle \mA, e_E\rangle$; so, $C^*_r\langle \mA, e_E\rangle$ is
  a unital $C^*$-algebra. Also, by the preceding theorem, there exists
  a (dual) conditional expectation $E_1 : C^*_r\langle \mA, e_E\rangle
  \to \mA$ with finite Watatani index.  Thus, it follows from the
  Proposition on Page 118 of \cite{STR} that $\mA$ is unital (in fact,
  $1_\mA$:= $E_1(\mathrm{id}_{\mathfrak{A}})$ works as the identity of
  $\mA$).  Since $E_1$ has finite Watatani index, it follows from
  \Cref{original-W-index} that it has a quasi-basis.
\end{proof}

\begin{remark}
With notations as in \Cref{dual-ce-izumi}, iterating the process of
$C^*$-basic construction and taking dual conditional expectations (as
above), one obtains a tower of $C^*$-algebras
\[
\mB \subset \mA \subset \mA_1 \subset \mA_2 \subset \cdots \subset
\mA_k \subset \cdots,
\]
where, for $k \geq 0$, $\mA_{k+1}:= C^*\la \mA_k, e_{k+1} \ra$ and
$e_{k+1}\in \mL_{\mA_{k-1}}( \mA_{k})$ is the Jones projection
corresponding to the finite-index (dual) conditional expectation $E_k: \mA_k
\to \mA_{k-1}$, with $\mA_{-1}:=\mB$, $\mA_0:= \mA$,  $E_0:=E$  and  $e_1 :=
e_E$.
\end{remark}

\section {Reduced $W^*$-basic construction}\label{W*-basic-construction}

Throughout this section, $\mN \subset \mM$ will denote a fixed
 inclusion of $W^*$-algebras with common identity and $E : \mM \to
 \mN$ a (fixed) faithful normal conditional expectation. Thus, as in
 \Cref{basic-construction}, $\mM$ becomes a (right) pre-Hilbert
 $C^*$-module over $\mN$ with respect to the natural $\mN$-valued
 inner product induced by $E$.  Let $\mathfrak{M}$ denote the
 self-dual completion of $\mM$ (see \Cref{self-dual-completion}) and
 $\eta_E:\mM \to \mathfrak{M}$ denote the canonical embedding. Thus,
 $\mathfrak{M}$ is a (right) Hilbert $W^*$-module over $\mN$; and,
 $\mL_\mN(\mathfrak{M})$ is a $W^*$-algebra, by
 \Cref{self-dual-facts}. The following fundamental facts will be
 crucial for our discussions ahead:

 \begin{theorem} [{\cite[Prop. 2.8  $ \& $ Thm. 3.2]{BDH}}]\( \)
   
    \begin{enumerate}
     \item There exists an injective (normal) $*$-homomorphism $\lambda_E:\mM
       \to \mL_\mN(\mathfrak{M})$ satisfying \(
       \lambda_E(x)(\eta_E(z)) = \eta_E(xz)$ for all $x, z \in \mM \). (Thus, 
$\lambda_E$ maps $\mM$ onto a $W^*$-subalgebra of
$\mL_\mN(\mathfrak{M})$.)
\item There exists a projection $e_E \in \mL_\mN(\mathfrak{M})$
  satisfying $e_E(\eta_E(x)) = \eta_E(E(x))$ for all $x \in \mM$.
\item The central support of $e_E$ in $\mL_\mN(\mathfrak{M})$ is $1$.
  \item  $\lambda_E(\mN) = \{e_E\}'\cap \lambda_E(\mM)$ in $\mL_\mN(\mathfrak{M})$. 
              \end{enumerate}
\end{theorem}

In view of the preceding facts, the following analogue of Jones'
notion of basic construction (from \cite{jones}) is very natural - also
see \cite{Kosaki, BDH, watatani, popa}.

\begin{definition}
The reduced $W^*$-basic construction of the inclusion $\mN \subset
  \mM$ (with respect to the conditional expectation $E : \mM \to \mN$)
  is  defined as the  $W^*$-subalgebra of  $\mL_\mN(\mathfrak{M})$ 
  generated by $\lambda(\mM)e_E\lambda(\mM)$ and will be denoted by $
  W^*_r\la \mM, e_E\ra$, i.e.,
  \[
  W^*_r\la \mM, e_E \ra := \overline{\mathrm{span}\{ \lambda(x) e_E
    \lambda(y) : x,y \in \mM\}}^{\sigma\mathrm{\text{-}weak}}
  \subseteq \mL_\mN(\mathfrak{M}).
  \]
  \end{definition}

  \begin{remark}
     When $\Ind_\p(E) <\infty$, Popa considers (in \cite[$\S
       1.1.3$]{popa}) the standard representation $\mM \subseteq
     B(\mH)$ of $\mM$ with the projection $e\in B(\mH)$ satisfying
     $exe=E(x)e$ for all $x\in \mM$ and $\overline{\mathrm{span}}(\mM
     e \mH) = \mH$, and goes on to define the basic construction of
     the inclusion $\mN \stackrel{E}{\subset} \mM$ as the von Neumann
     algebra $\{\mM \cup \{e\}\}''$ in $B(\mH)$.

     In fact, people also work with a more general notion of the
     basic construction of an inclusion $\mN \subset \mM$ without any
     normal conditional expectation from $\mM$ onto $\mN$ - see
     \cite[$\S 2.2$]{BMO}.

     However, keeping in mind our requirements and the obvious
     convenience of working with Hilbert $W^*$-modules, we prefer the
     Hilbert module approach, which was employed quite satisfactorily in
     \cite{BDH} and \cite{watatani}.
    \end{remark}

  \begin{remark}
In view of \Cref{compact-operators}, it follows that $W^*_r\la \mM, e_E\ra$ is a
$\sigma$-weakly closed two-sided ideal in $\mL_\mN(\mathfrak{M})$.
\end{remark}

  \begin{theorem}[\cite{Pas, BDH,  FK}]  \label{FiniteE}
The following statements are equivalent:
  \begin{enumerate}
  \item $\Ind_\p(E) <
    \infty$.
  \item $\Ind_\cp(E) < \infty$.\hfill (\cite{FK} - see \Cref{ce-facts} above)
  \item $\mM$ is a self-dual Hilbert $C^*$-module over $\mN$ (thus,
    $\mathfrak{M} =\eta_E(\mM)$).\hfill (\cite[Prop. 3.3]{BDH})

\item There exists a family $\{m_j\}$ in $\mM$ such that $\{\eta_E
  (m_j)\}$ is an orthonormal basis for the Hilbert $W^*$-module
  $\mathfrak{M}$ (over $\mN$) and the family $\{m_j m_j^*\}$ is
  $\sigma$-weakly unconditionally summable in $\mM$.\hfill
  (\cite[Thm. 3.12]{Pas},\cite[Thm. 3.5]{BDH})
  \end{enumerate}
In particular, if $\Ind_\p (E) < \infty$, then $\mM$ is
    a Hilbert $W^*$-module over $\mN$ and $\mathcal{L}_{\mN}(\mM)$ is a
  $W^*$-algebra.
  \end{theorem}
  (Notice that Item (4)  is comparable with  \cite[Thm. 1.1.6]{popa}.)

  \begin{remark}  \label{eta-bi-continuous} \label{eta-w*-continuous}
    \begin{enumerate}
      \item
 Clearly, $\eta_E$ is a linear contraction. And, $\eta_E: \mM \to
 \eta_E(\mM)$ is bi-continuous with respect to the norms if and only
 if $\Ind_\p(E)< \infty$ - see \cite[Prop. 3.3]{BDH}.
  \item If $\Ind_p(E) < \infty$, then $\eta : \mM \to
 \eta_E(\mM)$ is
  ($\sigma$-weak, $w^*$)-continuous as well, where the
  $w^*$-topology on $\eta_E(\mM)$ is induced by the duality mentioned
  in \Cref{self-dual-facts}  - see \cite[Cor. 8.5.8]{BM}.
    \end{enumerate}
  \end{remark}
A priori, it is not clear   whether $W^*_r\la \mM,
 e_E\ra$ always equals $\mL_\mN(\mathfrak{M})$ or not.  However, when
$\Ind_\p(E)<\infty$, then it is known (from \cite{BDH}) that the
equality holds.

  \begin{theorem}\cite{BDH}\label{more-finite-E}
 If $\Ind_{\p}(E) < \infty$, then identifying $\mM$ with
 $\lambda_E(\mM)$, the following hold:
 \begin{enumerate}
   \item There exists a bounded normal operator
    valued weight $\widehat{E}$ from $ \mL_\mN(\mM)$ to $\mM$
    satisfying $\widehat{E}(x e_E x^*) = xx^*$
    for all $x\in \mM$. \hfill  (\cite[Thm. 3.5]{BDH})

  \item For any family $\{m_j\}$ in $\mM$ such that $\{\eta_E(m_j)\}$
    is an orthonormal basis for the Hilbert $W^*$-module $\eta_E(\mM)$
    over $\mM$,
    \begin{enumerate}
\item 
\( \sum_j m_j e_E m_j^* = 1\) ($\sigma$-weakly) in
$\mL_\mN(\mM)$; and, \hfill (\cite[ Rem. 3.4(iii)]{BDH})
    \item $\sum_j m_j m_j^* = \widehat{E} (1)$ ($\sigma$-weakly) in
      $\mM$ and the sum is, therefore, independent of the chosen
      orthonormal basis and is a positive invertible element in
      $\mZ(\mM)$.\\ (\cite[Thm. 3.5]{BDH}).
   \end{enumerate}
 \end{enumerate}
In particular, $W^*_r\la \mM, e_E\ra = \mL_\mN(\mM)$ and there exists
a dual normal conditional expectation $E_1 : \mathcal{L}_{\mN}(\mM)
\to \mM$ with $\Ind_\p(E_1)< \infty$ and satisfying $E_1(x e_E y)=
\widehat{E}(1)^{-1}xy$ for all $x, y \in \mM$.  \hfill (\cite[$\S
  3.10$]{BDH})
\end{theorem}

\begin{definition}\cite[Defn. 3.6]{BDH} Let $\Ind_{\p}(E) <\infty$.
\begin{enumerate}
\item The BDH-index of $E$ is defined as $\mathrm{Ind}_{0}(E) =
  \widehat{E}(1)$, which is a positive invertible element in
  $\mZ(\mM)$. 
\item $E$ is said to have strongly finite index
if there exists a finite orthonormal basis for the Hilbert
$W^*$-module $\eta(\mM)$ over $\mM$.
\end{enumerate}
\end{definition}
\noindent (Note: What we call the BDH-index of $E$ is simply called
the index of $E$ in \cite{BDH}.)

\begin{remark}
 If $\Ind_{\p}(E) < \infty$, then $W^*_r\la \mM, e_E\ra=\{
 \lambda(\mM) \cup \{e_E\}\}'' = \mL_\mN(\mM)$ and, as mentioned in
 \cite[$\S$ 3.10]{BDH}, the reduced $W^*$-basic construction for the
 inclusion $\mN \subset \mM$ can be iterated to obtain a tower of
 $W^*$-algebras
 \[
 \mN \subset \mM \subset \mM_1 \subset \cdots \subset \mM_k \subset
 \cdots,
 \]
with (dual) conditional expectations $E_k: \mM_k\to \mM_{k-1}$, $k
\geq 1$ satisfying $\Ind_\p(E_k)<\infty$ and $E_k(xe_{E_{k-1}}y) =
\Ind_0(E_{k-1})^{-1} xy$ for all $x,y\in \mM_{k-1}$, where
$\mM_{-1}:=\mN$, $\mM_0:=\mM$, $E_0:=E$ and $\mM_{k+1}:=
\mL_{\mM_{k-1}}(\mM_k) = \{\lambda(\mM_{k}) \cup\{ e_{E_k}\}\}''$ with
respect to the conditional expectation $E_k:\mM_k \to \mM_{k-1}$, for
all $k \geq 0$.
  \end{remark}
 
\begin{remark}\label{W*-b-c-finite-quasi-basis}
 If $E$ has a finite quasi-basis, then $\Ind_{\p}(E) < \infty$ and
 \[
 W^*_r\la \mM, e_E\ra =  \mathrm{span}\{
 \lambda(x) e_E \lambda(y): x, y \in \mM\} =  \mL_\mN(\mM)
  = C^*_r\la \mM, e_E\ra, \]
by \Cref{quasi-basis-facts}.
\end{remark}

 \subsection{Generalized quasi-bases}

Recall that, in \cite[$\S$ 1.1.4]{popa}, Popa
 discusses orthonormal bases for normal conditional expectations,
 wherein he uses the Bures topology for the convergence of the series
 involved in the definition. A well-known and desirable feature of
 Popa's notion was that for a crossed-product von Neumann algebra $\mM
 \rtimes G$ (with respect to an action of an infinite discrete group
 $G$ on $\mM$), one has $x = \sum_{g \in G} g E(g^*x)$ in the Bures
 topology (and not in the $\sigma$-weak topology) - see
 \cite{mercer}.

However, we noticed that a similar (not necessarily orthogonal) notion
with respect to the $\sigma$-weak topology fits better in the present
context.  Given it's utility, inspired by Watatani, we  
call it a `generalized quasi-basis'.
\begin{definition}
A family $\{m_j\}$ in $\mM$ will be called a generalized (right)
quasi-basis for $E$ if
\[
x =\sum_j^{\sigma\text{-}\mathrm{weak}} m_j E (m_j^*x)  \text{ for all } x \in
\mM.
\]
\end{definition}

The following fact is well-known to the experts - see, for instance,
\cite[Page 223]{Jolissant}. (It is comparable with Item (3) on
\cite[Page 4]{popa} as well.) Though, we could not trace an explicit
proof of the same in the literature. Thus, for the sake of
completeness, we provide a proof which depends on some useful
observations made in \cite{BDH}, \cite{Pas} and \cite{JSvnc}.

\begin{proposition} \cite{Jolissant}\label{convergence-in-M}
  If $\Ind_\p(E)< \infty$, then for every collection $\{m_j\}$ in
  $\mM$ such that $\{\eta(m_j)\}$ is an orthonormal basis for
  $\eta(\mM)$ over $\mN$, $\{m_j\}$ is a generalized quasi-basis for
  $E$.

  In particular, $E$ admits a generalized quasi-basis.
\end{proposition}

\begin{proof}
Let us fix one such orthonormal basis for $\eta(\mM)$, say,
$\{\eta(m_j) : j \in J\}$ (see \Cref{FiniteE} for the existence). 

Let $x \in \mM$. Consider $\mM$ as a Hilbert $C^*$-module over $\mM$
with respect to the natural $\mM$-valued inner product $\la x, y
\ra_\mM:= x^*y$, $x, y \in \mM$.  Let $\mF(J)$ denote the (directed)
set consisting of all finite subsets of $J$. Then, the nets
$\Big\{\sum_{j \in S} \la m_j^*, m_j^*\ra_\mM: S \in \mF(J)\Big\}$ and
$\Big\{ \sum_{j\in S} \la E(m_j^*x), E(m_j^*x)\ra_\mM : S \in \mF(J)
\Big\}$ are both norm-bounded in $\mM$ because
\[
\sum_{j\in S}  \la m_j^*, m_j^*\ra_\mM = \sum_{j\in S}
m_j m_j^* \leq \sum_j^{\sigma\text{-weak}} m_j m_j^* = \hat{E}(1);
\text{ and,}
\]
\[
\sum_{j\in S} \la E( m_j^*x), E(m_j^*x) \ra_\mM = \sum_{j \in S}
E( m_j^*x)^*E(m_j^*x)
\leq \sum_{j\in S} E\left( (m_j^*x)^*(m_j^*x)\right)
\leq E(x^* \hat{E}(1) x)
\]
for all $ S \in \mF(J)$, where the second-last inequality follows from
the fact that every conditional expectation is a Schwartz mapping -
see \cite[$\S 9.2$]{STR}.

Thus, it readily follows that the net $\Big\{ \sum_{j \in S} \la
m_j^*, E(m_j^*x)\ra_\mM : S \in \mF(J)\Big\}$ is norm-bounded in $\mM$
(and is convergent in the $\sigma$-weak topology) - see the discussion
made on Page 458 of \cite{Pas} (also see \cite[$\S 8.5.26$]{BM}).  We
assert that it converges $\sigma$-weakly to $x$.

  Notice that $\eta(x) = \sum^{w^*}_j \eta(m_j) \la
\eta(m_j), \eta(x)\ra_\mN$, where the $w^*$-topology comes via the
(Paschke's) duality mentioned in \Cref{self-dual-facts}. Also, we have
\[
\Big\|\sum_{j \in S}  \eta(m_j) \la
\eta(m_j), \eta(x)\ra_\mN \Big\| \leq \big\|\sum_{ j \in S} m_j E(m_j^* x)\big\|
\]
for all $S \in \mF(J)$; thus, the net $\big\{\sum_{j \in S} \eta(m_j) \la
\eta(m_j), \eta(x)\ra_\mN\big\}$ is norm-bounded in $\eta(\mM)$; so, by
\Cref{self-dual-facts} again,
\[
\varphi \Big(\Big\la \eta\big(\sum_{j \in S} m_j E(m_j^* x) \big) ,
\eta(y)\Big\ra_\mN\Big) \to \varphi(\la \eta(x), \eta(y)\ra_\mN)
\]
for all $\varphi \in \mN_*$ and $y \in \mM$, which is the same as saying that 
\begin{equation}\label{phi-y}
\varphi \Big( E \big( \sum_{j \in S}
E (x^*m_j)m_j^*y-x^*y \big)\Big)  \to 
0
\end{equation}
for all $\varphi \in \mN_*$ and $y \in \mM$.

For every $\varphi\in\mN_*$ and $y\in \mM$, let $\varphi_y: \mM \to \C$
be given by $\varphi_y(z) = \varphi\big(E(zy)\big)$, $z\in \mM$. Then,
$\varphi_y\in \mM_*$ and it is known that the set $L := \mathrm{span}
\{\varphi_{y} : y\in \mM, \varphi \in \mN_* \}$ is norm-dense in
$\mM_*$ - see \cite[Lemma 3.2]{JSvnc}. Thus,  (\ref{phi-y}) tells us that
\begin{equation}\label{L-cgt}
\psi \Big( \sum_{j \in S} E(x^* m_j)m_j^* - x^*\Big) \to 0 \text{ 
for all } \psi \in L.\end{equation}

For every $S \in \mF(J)$, let $x_S:= \sum_{j \in S} E(x^* m_j) m_j^*\  
 (= \sum_{j \in S} \la m_j^*, E(m_j^*x)\ra_\mM^* )$. From above, the net
$\{x_S\}$ is norm-bounded in $\mM$.  Hence, by the denseness of $L$ in
$\mM_*$, it follows from (\ref{L-cgt}) that \( \theta
(x_S - x^*) \to 0 \) for all $\theta \in \mM_*$.  Thus, $x_S \to x^*$
($\sigma$-weakly) in $\mM$, so that $\sum_j^{\sigma\text{-weak}} m_j
E(m_j^* x) = x$.
\end{proof}

\section{Angle between compatible intermediate $W^*$-subalgebras}\label{W*-angles}

Having understood some basic aspects of the theory of $W^*$-basic
construction for inclusions of $W^*$-algebras, analogous to the notion
of interior angle (introduced in \cite{BG2}) between compatible
intermediate $C^*$-subalgebras of an inclusion of unital
$C^*$-algebras with a finite-index conditional expectation, we
introduce a similar notion between compatible intermediate
$W^*$-subalgebras of an inclusion of $W^*$-algebras with an
appropriate conditional expectation.

Throughout this section, $\mN \subset \mM$ will denote an  inclusion of
$W^*$-algebras (with common identity) and $E: \mM \to \mN$ will denote
a (faithful) normal conditional expectation   with $\Ind_\p(E)<\infty$.

\subsection{Compatible intermediate $W^*$-subalgebras}
Following \cite{IW},  we make the following definition:
\begin{definition} \label{IMS-p}
  An intermediate $W^*$-subalgebra $\mP$ of the inclusion $\mN\subset
  \mM$ will be called $E$-compatible if there exists a normal
  conditional expectation $ F : \mM \to \mP$ \text{ such that } $
  E_{\restriction_{\mP}}\circ F =E$.

  Further, $\mathrm{IMS}(\mN,\mM,E)$
  will denote the collection of all $E$-compatible intermediate
  $W^*$-subalgebras of the inclusion $\mN \subset \mM$.
\end{definition}

\begin{remark}\label{IMS}
  Let $\mP \in \mathrm{IMS}(\mN, \mM, E)$ with respect to a
  conditional expectation $F: \mM \to \mP$.
    \begin{enumerate} 
          \item 
 Since  $\mathrm{Ind}_{\p}(E) < \infty$, it follows from \cite[Prop.
   3.14]{BDH} that $\mathrm{Ind}_{\p}(F) <\infty$ as well.
\item       Since a conditional expectation with finite probabilistic index is
    faithful, it follows readily (as on Page 471 of  \cite{IW}) that if
    $\mP \in \mathrm{IMS}(\mN,\mM,E)$ with respect to two
    conditional expectations $F$ and $G$, then $F = G$.
    \end{enumerate}
    \end{remark}

\begin{proposition}\label{dualE1}\label{IMS-facts}
Let $\mP \in \mathrm{IMS}(\mN,\mM,E)$ with respect to the conditional
expectation $F : \mM\to \mP$ and $G:=E_{\restriction_\mP}$.  Then, the
following hold:
\begin{enumerate}
\item There exists a natural unital embedding
  $\mathcal{L}_{\mP}(\mM) \subseteq \mathcal{L}_{\mN}(\mM)$ via which
  $\lambda_F(x)$ is mapped to $\lambda_E(x)$ for all $x \in \mM$.  In
  particular,  $e_F \in \mathcal{L}_{\mN}(\mM)$  and, via this embedding,
  $e_F(\eta_E(x)) = \eta_E(F(x))$ for all $x \in \mM$.

\item $e_F e_E = e_E = e_E e_F$ in $\mathcal{L}_{\mN}(\mM)$.

\item For any orthonormal basis $\{\eta_G(w_l)\}$ for the Hilbert
  $W^*$-module $\eta_G(\mP)$ over $\mN$ (as in
  \Cref{FiniteE}),
  \[
  \sum_l^{\sigma\text{-}\mathrm{weak}} \lambda_E(w_l) e_E
  \lambda_E(w_l)^* = e_F\ (\text{in } \mL_\mN(\mM)).
\]
\item $E_1(e_F) = \mathrm{Ind}_{0}(E)^{-1}
  \mathrm{Ind}_{0}(E_{\restriction_{\mP}})$, where $E_1:\mL_\mN(\mM)
  \to \mM$ is the dual conditional expectation of $E$ (as in
  \Cref{more-finite-E}).
\item If $E$ has strongly finite index, then every finite orthonormal
  basis for the Hilbert $W^*$-module $\eta_E(\mM)$ (over $\mN$) is a
  quasi-basis for the conditional expectation $E$.
\item If $E_{\restriction_\mP}$ has a finite quasi-basis, say,
  $\{\mu_j\}$, then \( \sum_j \lambda_E(\mu_j) e_E \lambda_E(\mu_j)^*
  = e_F\ (\text{in } \mL_\mN(\mM)). \)
\end{enumerate}
\end{proposition}
\begin{proof}
(1) and (2) are immediate  (see \cite[Prop. 2.7]{GS}). \smallskip
  
(3): Fix an orthonormal basis $\{\eta_G(w_l)\}$ for $\eta_G(\mP)$, as in
  the hypothesis. Then, 
  \[
  x = \sum_l^{\sigma\text{-weak}} w_l
  E_{\restriction_\mP} (w_l^* x)
  \]
  for all $x \in \mP$, by \Cref{convergence-in-M}. Further, for every
  $a \in \mM$ and a finite set $S $ in the indexing set of $\{w_l\}$,
  we have
\begin{eqnarray*}
	\Big(\sum_{l \in S} \lambda_E(w_l) e_E
        \lambda_E(w_l)^*\Big)(\eta_E(a)) &=& \eta_E\Big(\sum_{ l \in
          S} w_l E(w_l^* a)\Big) \\ &=& \eta_E\Big(\sum_{l \in S}
        w_l (E_{\restriction_\mP} \circ F) (w_l^* a)\Big) \\ &=&
        \eta_E\Big(\sum_{ l \in S} w_l E_{\restriction_\mP}(w_l^*
        F(a))\Big).
\end{eqnarray*}
Since $\eta$ is ($\sigma$-weak, $w^*$)-continuous (see
\Cref{eta-w*-continuous}(2)), it follows that
\[{ \Big(\sum_{l \in S} \lambda_E(w_l) e_E
          \lambda_E(w_l)^*\Big)(\eta_E(a)) \stackrel{w^*}{\to}
          \eta_E(F(a)) = e_F(\eta_E(a)) }
\]
in $\eta_E(\mM)$, for all $a \in \mM$. Further, the net of finite sums
$\big\{\sum_{l\in S} \lambda_E(w_l) e_E \lambda_E(w_l)^* \} $ is 
norm-bounded  in $\mM$ because
\[
\sum_{l\in S} \lambda_E(w_l) e_E \lambda_E(w_l)^* \leq \sum_{l\in S}
\lambda_E(w_l)  \lambda_E(w_l)^* = \lambda_E\Big(\sum_{l \in S} w_l
w_l^*\Big) \leq \lambda_E(\hat{G}(1))
\]
for every finite indexing set $S$. Hence, by
\Cref{self-dual-facts} again,
\[
\sum_{l}^{\sigma\text{-weak}} \lambda_E(w_l) e_E
\lambda_E(w_l)^* = e_F \ \mathrm{ in }\  \mL_\mN(\mM).
\]

(4) follows from (3).\smallskip

(5) follows from \Cref{convergence-in-M}.\smallskip

(6) follows on the lines of (3).
\end{proof}
 
\subsection{Interior and exterior angle}
{\em For the sake of convenience, we simplify some
  notations as follows:}

(1) Following Jones and Watatani, we write $\mM_1$ for the reduced
$W^*$-basic construction $W^*_r\la \mM, e_E\ra (= \mL_\mN(\mM))$.
Recall from \Cref{dualE1} that, for $\mP$, $\mQ \in \mathrm{IMS}(\mN,
\mM,E)$, we have $\mP_1 , \mQ_1 \subset \mM_1$, where (by a slight
abuse of notation) $\mP_1:= \mL_\mP(\mM)$ and $\mQ_1:=\mL_\mQ(\mM)$,
with respect to the unique conditional expectations that make them
compatible.

(2) Since the conditional expectations making $\mP$ and $\mQ$ compatible
are unique, we simply write $e_\mP$ and $e_\mQ$ for the corresponding
Jones projections in $\mP_1$ and $\mQ_1$, respectively. Likewise, we
write $e_\mN$ for $e_E$. Also, for any Hilbert $C^*$-module $\mX$ over
a $C^*$-algebra $\mA$ and $x\in \mX$, $\|x\|_\mX$ denotes the norm of
$x$ in $\mX$ (induced by the $\mA$-valued inner product). \smallskip

With this, we can now give the definitions of the interior and
exterior angles.

Consider the dual normal conditional expectation $E_1: \mM_1\to \mM$
as in \Cref{dualE1}. Then, with respect to the natural $\mM$-valued
inner-product on $\mM_1$ (induced by $E_1$), $\mM_1$ becomes a Hilbert
$W^*$-module over $\mM$ by \Cref{FiniteE} and, by the Cauchy-Schwartz
inequality for Hilbert $C^*$-modules - see \cite[Proposition
  2.3]{Pas}, we have
\[
\lVert \langle e_\mP- e_\mN, e_\mQ- e_\mN \rangle_{\mM} \rVert \leq
\lVert e_\mP -e_\mN \rVert_{\mM_1} \lVert e_\mQ -e_\mN \rVert_{\mM_1}
\]
for all $\mP, \mQ \in \mathrm{IMS}(\mN,\mM,E)$. (As mentioned
above, $\|\cdot\|_{\mM_1}$ denotes the Hilbert $C^*$-module norm on
$\mM_1$.) Thus, on the lines of \cite{BDLR} and \cite{BG2}, it is natural
to make the following definitions.
\begin{definition}\label{defn-W*-angle}
 The $W^*$-interior angle between a pair $\mP, \mQ \in
 \mathrm{IMS}(\mN,\mM,E)$, denoted by $\alpha_{W^*}(\mP,\mQ)$, is
 defined by the expression
	\begin{equation}\label{DefiWan}
\cos   (     \alpha_{W^*}(\mP,\mQ)) =  \frac{\lVert \langle e_\mP - e_\mN,
          e_\mQ - e_\mN \rangle_\mM \rVert }{\lVert e_\mP-e_\mN
          \rVert_{\mM_1} \lVert e_\mQ - e_\mN
          \rVert_{\mM_1}},
	\end{equation}
        with value in the interval $\left[0, \frac{\pi}{2}\right]$.

        Furthermore,
for $\mP, \mQ \in \mathrm{IMS}(\mN,\mM,E)\setminus \{\mN\}$ with
$\mP_1, \mQ_1 \in \IMS(\mM, \mL_\mN(\mM), E_1)$, the $W^*$-exterior angle
between $\mP$ and $\mQ$ is defined as
\begin{equation}\label{beta-angle}
  \beta_{W^*}(\mP, \mQ) =
\alpha_{W^*}(\mP_1, \mQ_1).
\end{equation}
\end{definition}
(When we need to keep track of the initial inclusion $\mN \subset \mM$,
we shall use the notation $\alpha_\mN^\mM(\mP,\mQ) $ for
$\alpha_{W^*}(\mP,\mQ) $.)

\begin{remark}
The preceding definition of interior angle agrees with the interior
angle defined between intermediate subfactors of an irreducible
finite-index subfactor $N \subset M$ defined in \cite{BDLR}. This
follows from the expressions for the interior angle given in
\cite[Prop. 2.14]{BDLR} and \Cref{W*-angle-quasi-basis} derived ahead.
\end{remark}

Analogous to \cite [Prop. 5.10(2)]{BG2}, employing
\Cref{IMS-facts}, the following expression for the interior angle is
immediate.
\begin{proposition}\label{W_angle}
Let  $\mP , \mQ \in \mathrm{IMS}(\mN,\mM,
E)$. Then,  
\begin{equation}\label{angle-expression}
\cos(\alpha_{W^*}(\mP,\mQ)) = \frac{\lVert E_1(e_\mP e_\mQ -
  e_\mN)\rVert}{\lVert \mathrm{Ind}_{0}(E)^{-1}
 ( \mathrm{Ind}_{0}(E_{\restriction_{\mP}})-1)\rVert^{1/2}\, \lVert
 \mathrm{Ind}_{0}(E)^{-1} ( \mathrm{Ind}_{0}(E_{\restriction_{\mQ}})-1) \rVert^{1/2}}.
\end{equation}

In particular, $\alpha(\mP, \mQ) = \frac{\pi}{2}$ if the quadruple
$(\mN, \mP, \mQ, \mM)$ is a commuting square (i.e., $e_\mP e_\mQ =
e_\mN$).
\end{proposition}

\begin{remark}
In \cite{BDLR}, for a quadruple of $II_1$-factors, the analogue of the
assertion in the preceding proposition regarding the quadruple being a
commuting square was a necessary as well as a sufficient
condition. There, the tracial property of the unique tracial state on
the basic construction $II_1$-factor played a crucial role in proving
the converse.
  \end{remark}

Employing \Cref{IMS-facts}(6), the following observation is immediate
(and is analogous to \cite[Prop. 3.3]{GS} and \cite[Prop.
  2.14]{BG2}).
\begin{proposition}\label{W*-angle-quasi-basis}
  Let $E$ have a finite quasi-basis; let $\mP, \mQ \in \IMS(\mM, \mN,
  E)$ with compatible conditional expectations $F: \mM \to \mP$ and
  $G: \mM \to \mQ$. Then, $F$ and $G$ have finite quasi-bases, say,
  $\{\mu_j\}$ and $\{\delta_k\}$, respectively; and,
  \begin{equation}
\cos\left(\alpha_{W^*}(\mP, \mQ)\right) = \frac{\|
  \Ind_W(E)^{-1}\left(\sum_{j,k}\mu_j E(\mu_j^* \delta_k)
  \delta_k^*-1\right) \|}{\| \Ind_W(E)^{-1}\left(\mathrm{Ind_W}
  (E_{\restriction_{\mP}}) -1\right) \|^{1/2}\,
  \|\Ind_W(E)^{-1}\left(\mathrm{Ind_W} (E_{\restriction_{\mQ}})-
  1\right)\|^{1/2}}.
\end{equation}
In particular, when $\Ind_W(E)$ is a scalar, then 
  \begin{equation}\label{ind-scalar-angle}
\cos\left(\alpha_{W^*}(\mP, \mQ)\right) =   \frac{\|\sum_{j,k}\mu_j
    E(\mu_j^* \delta_k)\delta_k^* -1
    \|}{\| \mathrm{Ind_W} (E_{\restriction_{\mP}}) -1
 \|^{1/2}\, \|\mathrm{Ind_W}
      (E_{\restriction_{\mQ}})- 1\|^{1/2}}.
\end{equation}
\end{proposition}
\begin{proof}
That $F$ and $G$ have finite quasi-bases follows from
\cite[Prop. 3.5]{Khos}. Rest follows verbatim on the lines of the
proofs of  \cite[Prop.  2.14]{BG2} and \cite[Prop. 3.3]{GS}.
  \end{proof}

\begin{remark}
  As expected, the notion of $C^*$-interior angle between compatible
  intermediate $C^*$-sualgebras of a finite-index inclusion of unital
  $C^*$-algebras (as defined in \cite[Definition 5.1]{BG2}) agrees
  with the above defined $W^*$-interior angle between their
  corresponding biduals - see \Cref{C-star-vs-W-star-angle}.
  \end{remark}

\begin{example}
 Let $G$ be a countable discrete group acting on a von Neumann algebra
 $\mM$, $H$ be a subgroup of $G$. Then, there is a natural embedding
 of $\mM \rtimes H$ into $\mM \rtimes G$ and there exists a unique
 faithful normal conditional expectation $E: \mM \rtimes G \to \mM
 \rtimes H$ satisfying
      \[
      E(\sum_{g\in G} x_g g) =
      \sum_{h\in H}x_h h
      \]
      for every $\sum_{g \in G}x_g g$ in the $\sigma$-weakly dense
      $*$-subalgebra $\left\{ \sum_{g\in G} x_{g} {g}: x_{g} \in
      \mM, g \in G \right \}$ (where the sums are finite) of $\mM
      \rtimes G$. Moroever, if $[G:H]<\infty$, then $E$ has a finite
      quasi-basis $\{ {g_i}: 1 \leq i \leq [G:H]\}$ for any set of
      coset representative $\{g_i: 1 \leq i \leq [G:H]\}$ of $H$ in
      $G$. In particular, $\Ind_W(E)=[G:H]$.

      Thus, if $[G:H]<\infty$ and $K$ and $L$ are two intermediate
      subgroups of the inclusion $H \subset G$, then $\mM \rtimes K,
      \mM \rtimes L \in \IMS(\mM \rtimes G, \mM \rtimes H, E)$ with
      respect to the (above-mentioned) canonical conditional
      expectations and, by \Cref{ind-scalar-angle} - after imitating the
      derivation of \cite[Prop. 5.2]{GS} - we obtain
  \[
\cos(\alpha_{W^*}(\mM \rtimes K, \mM \rtimes L)) = \frac{[K \cap L: H]
  -1}{\sqrt{[K:H]-1}\sqrt{[L:H]-1}}.
  \]
This generalizes the expression for the interior angle between
intermediate crossed-product subfactors obtained in
\cite[Prop. 2.7]{BG1}.
\end{example}  

\section{Angle between  compatible intermediate subalagebras of
  inclusions of non-unital $C^*$-algebras}\label{non-unital-angle}

We shall now apply the preceding notion of angle defined between
compatible intermediate $W^*$-subalgebras to define the notion of
interior angle between compatible intermediate $C^*$-subalgebras of
``finite-index'' inclusions of non-unital $C^*$-algebras

Throughout this section, $\mB \subset \mA$ will
denote a fixed inclusion of $C^*$-algebras with a conditional
expectation $E: \mA \to \mB$ with finite Watatani index
(\Cref{izumi-defn}). The set of $E$-compatible intermediate
$C^*$-subalgebras of the inclusion $\mB \subset \mA$ will be denoted
by $\mathrm{IMS}_W(\mB, \mA, E)$, i.e.,
\begin{eqnarray*} &
\mathrm{IMS}_W(\mB, \mA, E):=\Big\{\mC: \mC\ \text{is an intermediate
} C^*\text{-subalgebra of } \mB \subset \mA \text{ and there exists 
  a  } \\ &
 \text{ conditional    expectation } F:
\mA \to \mC \text{ such that } \Ind_W(F)< \infty \text{ and }
    {E}_{\restriction_{\mC}} \circ F = E\Big\}.
\end{eqnarray*}

As in \Cref{IMS}, if $\mC \in \mathrm{IMS}_W(\mB, \mA,
E)$ with respect to two conditional expectations $F, G: \mA \to \mC$,
then $F = G$.

From \Cref{dualE1}, we obtain the following useful
observations in the context of $C^*$-algebras. 

\begin{proposition}\label{wstarims}
Let $(E^{**})_1: \mL_{\mB^{**}}(\mA^{**}) \to \mA^{**}$ denote the
dual of the normal conditional expectation $E^{**}: \mA^{**} \to
\mB^{**}$ (as in \Cref{more-finite-E}) and $\mC \in \IMS_W(\mB,\mA,E)$
with respect to the conditional expectation $F: \mA \to C$ with finite
Watatani index. Then, the following hold:
	\begin{enumerate}
        \item $\mC^{**} \in \mathrm{IMS}(\mB^{**}, \mA^{**}, E^{**})$
          with respect to the conditional expectation $F^{**} :
          \mA^{**} \to \mC^{**}$.
		\item  $\mL_{C^{**}}(A^{**}) \subseteq \mL_{\mB^{**}}(A^{**})$.
	
		\item  $e_{\mC^{**}}e_{\mB^{**}} = e_{\mB^{**}}= e_{\mB^{**}}e_{\mC^{**}}$.
		\item There exists a family $\{ w_l\}$ in
                  $\mC^{**}$ such that $e_{\mC^{**}} =
                  \sum_l^{\sigma\text{-}\mathrm{weak}} w_l
                  e_{\mB^{**}} w_l^*$ in $\mL_{\mB^{**}}(A^{**})$.  
                \item $(E^{**})_1(e_{\mB^{**}}) =
                  \mathrm{Ind}_0(E^{**})^{-1} \in \mZ (A^{**})$.
		\item $(E^{**})_1(e_{\mC^{**}}) =
                  \mathrm{Ind}_0(E^{**})^{-1}
                  \mathrm{Ind}_0(E^{**}_{\restriction_{\mC}})$.\item
                  If $E_{\restriction_{\mC}}$ has a finite
                  quasi-basis, say, $\{\mu_j\}$, then $e_{\mC^{**}} =
                  \sum_j^{\sigma\text{-}\mathrm{weak}} \mu_j
                  e_{\mB^{**}} \mu_j^*$ in $\mL_{\mB^{**}}(A^{**})$.
	\end{enumerate}
\end{proposition}
 
\begin{proof}
(1) Notice that $\mB^{**}\subseteq \mC^{**} \subseteq \mA^{**}$ with
  common identity (see \Cref{E-to-E**}). Also,
  $(E^{**})_{\restriction_{C^{**}}} = (E_{\restriction_{C}})^{**}$ and
  $(E^{**})_{\restriction_C} \circ F^{**} = E^{**}$, i.e., $ \mC^{**}
  \in \mathrm{IMS}(\mB^{**}, \mA^{**}, E^{**})$.

Others follow readily from \Cref{dualE1}.
\end{proof}

On expected lines, it turns out that the $C^*$-interior angle
defined in \cite[Defn. 5.1]{BG2} between compatible
intermediate $C^*$-subalgebras of a finite-index inclusion of unital
$C^*$-algebras agrees with the $W^*$-interior angle between their
corresponding enveloping von Neumann algebras. 
\begin{proposition}\label{C-star-vs-W-star-angle}
  Let $ \mB\subset \mA$ be an inclusion of  unital $C^*$-algebras with a 
  conditional expectation $E :\mA \to \mB$ with finite Watatani
  index and let $\mC, \mD \in \IMS_W(\mA, \mB, E)$. Then,
  \[
\alpha_{W^*}(\mC^{**},\mD^{**}) = \alpha_{C^*}(\mC,\mD),
\]
where $\alpha_{C^*}$ denotes  the $C^*$-interior angle as in 
 \cite[Defn. 5.1]{BG2}.
\end{proposition}

\begin{proof}
  Note that $ \Ind_\cp(E^{**}) = \Ind_\cp(E)<\infty$,
  $\Ind_0(E^{**}) = \Ind_W(E^{**}) = \Ind_W(E)$ and $(E^{**})_1(e_{\mB^{**}}) =
  \Ind_0(E^{**})^{-1}$, where $(E^{**})_1: \mL_{\mB^{**}}(\mA^{**})
  \to \mA^{**}$ is the dual of the normal conditional expectation
  $E^{**}: \mA^{**} \to \mB^{**}$ - see \Cref{dualE1}.
  
Let $F: \mA \to \mC$ and $G: \mA \to \mD$ be the compatible
conditional expectations with finite Watatani indices.  By
\Cref{original-W-index}, all of $E$, $F$ and $G$ have finite
quasi-bases in $\mA$. Further, $E_{\restriction_{\mC}}$ and
$E_{\restriction_{\mD}}$ also have finite quasi-bases, say,
$\{\mu_j\}$ and $\{\delta_k\}$, respectively - see \cite{watatani} or
\cite[Prop. 2.7]{GS}. Thus, by \cite[Prop. 3.3]{GS}, the
$C^*$-interior angle between $\mC$ and $\mD$ is given by
\[
\cos(\alpha_{C^*}(\mC,\mD)) = \frac{\lVert
          \mathrm{Ind_W} (E)^{-1}\left(\sum_{j,k} \mu_j E(\mu_j^*
          \delta_k)\delta_k^* -1\right) \rVert}{\lVert
          \mathrm{Ind}_{W}(E)^{-1}\left(
          \mathrm{Ind}_{W}(E_{\restriction_{C}})-1\right)\rVert^{1/2} \lVert
          \mathrm{Ind}_{W}(E)^{-1}
          \left(\mathrm{Ind}_{W}(E_{\restriction_{D}})-1 \right)\rVert^{1/2}}.
\]

Next, notice that $\{\mu_j\}$ and $\{\delta_k\}$, respectively, are
quasi-bases for $(E^{**})_{\restriction_{\mC^{**}}} (=
(E_{\restriction_{\mC}})^{**})$ and
$(E^{**})_{\restriction_{\mD^{**}}} (=
(E_{\restriction_{\mD}})^{**})$, as well. Thus, from
\Cref{W*-angle-quasi-basis}, it follows that
\begin{eqnarray*}
  \lefteqn{ 	\cos(\alpha_{W^*}(\mC^{**}, \mD^{**}))}\\
  & = & \frac{\lVert
          \mathrm{Ind}_W (E^{**})^{-1} \left( \sum_{j,k} \mu_j
          E^{**}(\mu_j^* \delta_k)\delta_k^* -1\right) \rVert}{\lVert
          \mathrm{Ind}_W (E^{**})^{-1} \big( \mathrm{Ind}_{W}
          \left((E^{**})_{\restriction_{C^{**}}}\right) - 1 \big) \rVert^{1/2}
          \lVert \mathrm{Ind}_W (E^{**})^{-1}
          \big(\mathrm{Ind}_{W}\left( (E^{**})_{\restriction_{D^{**}}}\right)-1
            \big)\rVert^{1/2}
  }\\ & = &
\frac{\lVert
          \mathrm{Ind_W} (E)^{-1}\left(\sum_{j,k} \mu_j E(\mu_j^*
          \delta_k)\delta_k^* -1\right) \rVert}{\lVert
          \mathrm{Ind}_{W}(E)^{-1}\left(
          \mathrm{Ind}_{W}(E_{\restriction_{C}})-1\right)\rVert^{1/2} \lVert
          \mathrm{Ind}_{W}(E)^{-1}
          \left(\mathrm{Ind}_{W}(E_{\restriction_{D}})-1 \right)\rVert^{1/2}},
\end{eqnarray*} and we are done.
\end{proof}

In view of \Cref{C-star-vs-W-star-angle},  the following
definition is quite natural.
\begin{definition}
  Let $ \mB\subset \mA$ be an inclusion of  $C^*$-algebras with
  a conditional expectation $E :\mA \to \mB$ with finite Watatani index.
   Then, the $C^*$-interior angle
  between any two  $\mC, \mD \in \IMS_W(\mA, \mB, E)$ is defined by
\begin{equation}
	\alpha_{C^*} (\mC,\mD)) = \alpha_{W^*}(\mC^{**},\mD^{**}).
\end{equation}
\end{definition}

\begin{remark}
Prior to the preceding definition, Bakshi et al (\cite{BGJ}) gave a
definition for the interior angle between intermediate
$C^*$-subalgebras of an irreducible inclusion of $\sigma$-unital
simple $C^*$-algebras. In view of an observation preceding
\cite[Thm. 2.8]{izumi}, wherein Izumi illustrates that the Jones
projection $e_{\mB^{**}}$ can be identified with $e_\mB$, both definitions
agree under the hypothesis of \cite[Defn. 2.4]{BGJ}.
  \end{remark}

\subsection{Stability of the interior angle under the minimal tensor product}

Recall that for any $C^*$-algebra $\mP$ and any inclusion of
$C^*$-algebras $\mB \subset \mA$, there is a natural $C^*$-embedding
of $\mB \omin \mP$ into $\mA \omin \mP$ and if $E: \mA \to \mB$ is a
conditional expectation, then $E \omin \mathrm{id}_\mP: \mA \omin \mP
\to \mB \omin \mP$ is also a conditional expectation.

 In \cite{BGJ}, it was proved that for any irreducible
inclusion $\mB \subset \mA$ of simple unital $C^*$-algebras with a
conditional expectation $E: \mA \to \mB$ of finite Watatani index,
\[
\alpha(\mC \omin \mK(\ell^2), \mD \omin \mK(\ell^2)) = \alpha (\mC,
\mD)\] for every pair of intermediate $C^*$-subalgebras $\mC$ and
$\mD$ of the inclusion $\mB \subset \mA$. Interestingly, for unital
$C^*$-algebras, a more general stability result holds and the proof is
much simpler.

\begin{proposition}
Let $\mB \subset \mA$ be an inclusion of unital $C^*$-algebras with a
conditional expectation $E :\mA \to \mB$ of finite Watatani index and
$\mP$ be a fixed unital $C^*$-algebra. Then, the following hold:
\begin{enumerate}
\item The conditional expectation $E \omin \mathrm{id}_\mP: \mA \omin
  \mP \to \mB \omin \mP $ has finite Watatani index and $\Ind_W(E
  \omin \mathrm{id}_\mP) = \Ind_W(E) \otimes 1$.
\item If $\mC  \in
\mathrm{IMS}_W(\mB, \mA, E)$, then $\mC \omin \mP \in
\mathrm{IMS}_W(\mB \omin \mP , \mA \omin \mP, E \omin
\mathrm{id}_\mP)$.
\item If $\mC, \mD \in
\mathrm{IMS}_W(\mB, \mA, E)$, then \[
\alpha_{C^*} (\mC \omin \mP, \mD \omin \mP) = \alpha_{C^*}
(\mC,\mD).
\]
\end{enumerate}\end{proposition}
\begin{proof}
  (1) and (2) follow  immediately from the definitions.\smallskip

  (3): Let $(E\otimes \mathrm{id}_\mP)_1: (\mA \omin
\mP)_1 \to \mA \omin \mP$ denote the dual of the conditional
expectation $E \omin \mP$.  Then, $ (E \otimes \mathrm{id}_\mP)_1
(e_{\mB \omin \mP}) = \mathrm{Ind}_W (E\omin 1)^{-1} = \mathrm{Ind}_W
(E)^{-1} \ot 1 $.

Let $\{\mu_j\}$ and $\{\delta_k\}$ be finite
quasi-bases for $E_{\restriction_\mC}$ and $E_{\restriction_\mD}$,
respectively. Since $(E \omin \mathrm{id}_\mP)_{\restriction_{\mC
    \omin \mP}} = (E_{\restriction_{\mC}} \omin \mathrm{id}_\mP)$, it
follows that $\{\mu_j \otimes 1\}$ is a (finite) quasi-basis for $ (E
\omin \mathrm{id}_\mP)_{\restriction_{\mC \omin \mP}}$. Likewise,
$\{\delta_k\otimes 1\}$ is a (finite) quasi-basis for $ (E \omin
\mathrm{id}_\mP)_{\restriction_{\mD \omin \mP}}$.  Thus, it follows
from \cite[Prop. 2.7]{GS} that
\[
e_{\mC \omin \mP} = \sum_j (\mu_j \otimes 1) e_{\mB \omin
  \mP} (\mu_j^{*}\otimes 1) \text{ and }
e_{\mD
  \omin \mP} = \sum_k (\delta_k\otimes 1) e_{\mB \omin \mP}
(\delta_k^{*}\otimes 1); \text{ so,}
\]
 \begin{eqnarray*} \lefteqn{(E \omin \mathrm{id}_\mP)_1
    (e_{\mC\omin \mP} e_{\mD \omin \mP})}\\ & = & (E \omin
  \mathrm{id}_\mP)_1 \left( \sum_{j,k} (\mu_j \otimes 1) e_{\mB \omin
    \mP} (\mu_j^{*}\delta_k\otimes 1)e_{\mB \omin \mP}
  (\delta_k^{*}\otimes 1) \right)\\ & = & (E \omin \mathrm{id}_\mP)_1
  \left( \sum_{j,k} (\mu_j \otimes 1) (E \omin \mathrm{id}_\mP)
  (\mu_j^{*}\delta_k\otimes 1)e_{\mB \omin \mP} (\delta_k^{*}\otimes
  1) \right)\\ & = & (E \omin \mathrm{id}_\mP)_1 \left( \sum_{j,k}
  (\mu_j \otimes 1) \left(E (\mu_j^{*}\delta_k) \otimes 1\right)e_{\mB
    \omin \mP} (\delta_k^{*}\otimes 1) \right)\\ & = & \Ind_W(E \omin
  \mathrm{id}_\mP)^{-1} \left( \sum_{j,k} (\mu_j \otimes 1) \left(E
  (\mu_j^{*}\delta_k) \otimes 1\right) (\delta_k^{*}\otimes 1) \right)
  \\ & = & \big( \Ind_W(E)^{-1} \ot 1 \big) \left( \sum_{j,k} \mu_j E
  (\mu_j^{*}\delta_k) \delta_k^* \otimes 1 \right),
 \end{eqnarray*}
 \begin{eqnarray*}
\lefteqn{ \left\| e_{\mC\omin \mP} - e_{\mB\omin \mP}\right\|_{(\mA
    \omin \mP)_1}}\\ & = & \left\| (E\omin \mathrm{id}_\mP)_1\big(
(e_{\mC\omin \mP}- e_{\mB\omin \mP})^* (e_{\mC\omin \mP}-e_{\mB\omin
  \mP})\big)\right\|^{1/2} \\&=& \left\| (E\omin
\mathrm{id}_\mP)_1(e_{\mC\omin \mP}- e_{\mB\omin \mP})\right\|^{1/2}
\\ &=& \Big\| \left(E\omin \mathrm{id}_\mP\right)_1\Big(\sum_{ j}
(\mu_j\otimes 1) e_{\mB\omin \mP}(\mu_j^*\otimes 1)-e_{\mB\omin \mP}
\Big)\Big\|^{1/2} \\ &=& \Big\| \sum_{ j}(\mu_j\otimes 1) (E\omin
\mathrm{id}_\mP)_1(e_{\mB \omin \mP}) (\mu_j^*\otimes 1) -(E\omin
\mathrm{id}_\mP)_1(e_{\mB\omin \mP})\Big\|^{1/2} \\ &=& \Big\| 
\sum_{ j} (\mu_j\otimes 1)(\mu_j^*\otimes
1)\big(\mathrm{Ind_W}(E)^{-1}\otimes 1  \big)
-\mathrm{Ind_W}(E)^{-1}\otimes 1\Big\|^{1/2} \\ &=& \Big\| \Big( \sum_{ j}
\mu_j \mu_j^* \ot 1 -1\ot 1\Big)\big(\mathrm{Ind_W}(E)^{-1}\otimes 1 \big) 
\|^{1/2}\\  & = & \|
\left(\mathrm{Ind_W}( E)^{-1} (\mathrm{Ind_W (E_{\restriction_{\mC}})}
- 1)\right)\otimes 1 \|^{1/2};
          \end{eqnarray*}
 and, likewise,
 \[
\| e_{\mD \omin
      \mP} - e_{\mB \omin \mP}\|_{(\mA \omin \mP)_1}
= \|
    \mathrm{Ind_W} (E)^{-1} \left(\mathrm{Ind_W (E_{\restriction_{\mD}})}-
1 \right)\otimes 1 \|^{1/2}.
 \]
Hence,    
   \begin{eqnarray*}
  \lefteqn{ \cos (\alpha_{C^*} ( \mC\omin \mP, \mD \omin \mP))} \\ & =
  & \frac{\| \langle e_{\mC\omin \mP}- e_{\mB\omin \mP}, e_{\mD \omin
      \mP}-e_{\mB \omin \mP} \rangle_{\mA\omin \mP}\|}{\| e_{\mC\omin
      \mP} - e_{\mB\omin \mP}\|_{\mA_1\omin \mP}\ \| e_{\mD \omin \mP}
    - e_{\mB\otimes P}\|_{\mA_1 \omin \mP} }\\ &= & \frac{\| (E \omin
    \mathrm{id}_\mP)_1 (e_{\mC\omin \mP} e_{\mD \omin \mP}) -
    (E\otimes \mathrm{id}_\mP)_1(e_{\mB \omin \mP}) \| }{\|
    e_{\mC\omin \mP} - e_{\mB\omin \mP}\|_{\mA_1 \omin \mP}\ \| e_{\mD
      \omin \mP} - e_{\mB\otimes \mP}\|_{\mA_1 \omin \mP}}\\&= &
  \frac{\| (\mathrm{Ind_W} (E)^{-1} \sum_{j,k}\mu_j E(\mu_j^*
    \delta_k)\delta_k^* )\otimes 1 - \mathrm{Ind_W} (E)^{-1}\otimes 1
    \|}{ \|
    \mathrm{Ind_W}( E)^{-1} \left(\mathrm{Ind_W (E_{\restriction_{\mC}})} -
    1\right)\otimes 1 \|^{1/2}\|
\|
    \mathrm{Ind_W}( E)^{-1} \left(\mathrm{Ind_W (E_{\restriction_{\mD}})} -
    1\right)\otimes 1 \|^{1/2}\|
} \\ &=
  & \frac{\| (\Ind_W(E)^{-1})(\sum_{j,k}\mu_j E(\mu_j^*
    \delta_k)\delta_k^* -1) \|}{\| (\Ind_W(E)^{-1})(\mathrm{Ind}_W
      (E_{\restriction_{\mC}}) -1) \|^{1/2}
    \|(\Ind_W(E)^{-1})(\mathrm{Ind}_W (E_{\restriction_{\mD}})- 1)\|^{1/2}}
  \\ &= & \cos(\alpha_{C^*}(\mC,\mD)),
   \end{eqnarray*}
where the last equality follows from  \cite[Proposition 3.3]{GS}.
\end{proof}
\noindent{\bf Question:} It will be interesting to see whether a
similar stability of the interior angle under minimal tensor product
holds for general non-unital $C^*$-algebras or not.\medskip

\noindent {\bf Acknowledgements:} The first named author would like to thank
  Keshab Chandra Bakshi and Biplab Pal for some  insightful
  discussions on \cite{izumi}.


\begin{thebibliography}{99}

\bibitem{BDH} M.~Baillet, Y.~Denizeau and J-F.~Havet, Indice d'une
  esp\'{e}rance conditionnelle, {\it Compositio Math.} {\bf 66} (1988),
  no.~2, 199–236.

\bibitem{BDLR} K.~C.~Bakshi, S.~Das, Z.~Lieu and Y.~Ren, An angle
  between intermediate subfactors and its rigidity,
  {\it Trans.~Amer.~Math.~Soc.} {\bf 371} (2019), no.~8, 5973-5991.
  
  \bibitem{BGJ} K.~C.~Bakshi, S.~Guin and D.~Jana, A few remarks on
    intermediate subalgebras of an inclusion of $C^*$-algebras,
    {\it Infin. Dimens. Anal. Quantum Probab. Relat. Top.} {\bf 28} (2025),
    no.~4, 2550006.

  \bibitem{BG1} K.~C.~Bakshi and V.~P.~Gupta, {\it Internat.~J.~Math.} {\bf
    30} (2019), no.~12, 1950061.

\bibitem{BG2} K.~C.~Bakshi and V.~P.~Gupta, Lattice of intermediate subalgebras,
  {\it J.~London Math.~Soc.} {\bf 104} (2021), no.~5, 2082-2127.

\bibitem{BG3} K.~C.~Bakshi and V.~P.~Gupta, Regular inclusions of
  simple unital $C^*$-algebras, {\it M\"{u}nster J.~of Math.} {\bf 18}
  (2025), 181-200.

\bibitem{BMO} J.~Bannon, A.~Marrakchi and N.~Ozawa, Full factors and
  co-amenable inclusions, {\it Comm.~Math.~Phys.}~{\bf 378} (2020),
  1107-1121.

\bibitem{blackadar} B.~Blackadar, {\it Operator algebras: Theory of
  $C^*$-algebras and von Neumann algebras}, Springer-Verlag Berlin
  Heidelberg, 2006.

\bibitem{BM} D.~P.~Blecher, and C.~ Le Merdy, {\it Operator algebras and
  their modules: An operator space approach}, London Mathematical
  Society Monographs, Series 30.
  
\bibitem{JSvnc} J.~Cameron and R.~R.~Smith, Bimodules in crossed
  products of von Neumann algebras, {\it Adv.~Math.} {\bf 274}
  (2015), 539-561.

\bibitem{JCRS} J.~Cameron and R.~R.~Smith, A Galois correspondence for
  reduced crossed products of simple $C^*$- algebras by discrete
  groups, {\it Canad.~J.~Math.} {\bf 71} (2019), no.~5, 1103-1125.
  
\bibitem{FK} M.~Frank, and E.~Kirchberg, On conditional expectations
  of finite index, {\it J. Operator Theory} {\bf 40} (1998), 87–111.
  
\bibitem{GS} V.~P.~Gupta and D.~Sharma, On possible values of the
  interior angle between intermediate subalgebras,
  {\it J.~Aust.~Math.~Soc.} {\bf 117} (2024), no.~1, 44-66.

\bibitem{Haag} U.~Haagerup, Operator valued weights in von Neumann
  algebras, {\it J.~Funct.~Anal.} {\bf 32} (1979), 175 - 206.

\bibitem{IW} S.~Ino and Y.~Watatani, Perturbations of intermediate
    $C^*$-subalgebras for simple $C^*$-algebras, {\it Bull. London
    Math. Soc.} {\bf 46} (2014), 469-480.

\bibitem{izumi} M.~Izumi, Inclusions of simple $C^*$-algebras, {\it J. Reine
  Angew Math.} {\bf 547} (2002), 97-138.

\bibitem{Jolissant} P.~Jolissant, Indice d'esperances conditionelles
  et algebres de von Neumann finies, {\it Math. Scand.} {\bf 68}
  (1991), 221 - 246.

\bibitem{jones} V.~F.~R.~Jones, Index for subfactors, {\it
  Invent.~Math.} {\bf 72} (1983), 1-25.

\bibitem{KPW} T.~Kajiwara, C.~Pinzari and Y.~Watatani, Jones index
  theory for Hilbert $C^*$-bimodules and its equivalence with
  conjugation theory, {\it J. Funct. Anal.} {\bf 215} (2004), 1–49.

\bibitem{Khos} M.~Khoshkam, Hilbert $C^*$-modules and conditional
  expectation on crossed products, {\it J.~Aust.~Math.~Soc. (Series A)} {\bf
    61} (1996), 106-118.

  \bibitem{Kosaki} H.~Kosaki, Extension of Jones' theory on index to
    arbitrary factors, {\it J.~Funct.~Anal.} {\bf 66} (1986) 123-140.

    \bibitem{Lance} E.~C.~Lance, {\it Hilbert $C^*$-modules: a toolkit
      for operator algebraists}, LMS LNS {\bf 210}, Cambridge
      University Press, 1995.

    \bibitem{Longo} R.~Longo, Conformal subnets and intermediate
      subfators, {\it Commun. Math. Phys.} {\bf 237} (2003), 7-30.
      
    \bibitem{mercer} R.~Mercer, Convergence of Fourier series in
      discrete crossed products of von Neumann algebras, {\it
        Proc. Amer. Math. Soc.} {\bf 94} (1985), 254 - 258.

    \bibitem{Pas} W.~L.~Paschke, Inner product modules over
    $B^*$-algebras, {\it Trans.~Amer.~Math.~Soc.} {\bf 182} (1973), 443–468.

    \bibitem{Ped} G.~K.~Pedersen, {\it $C^*$-algebras and their
      automorphism groups}, Academic Press, 1979.

    \bibitem{popa} S.~Popa, Classification of subfactors and their
      endomorphisms, {\it CBMS Regional Conference series in
      Mathematics} {\bf 86}, Americal Mathematical Society, 1995.


\bibitem{Rieffel} M.~A.~Rieffel, Morita equivalence for
  $C^*$-algebras and $W^*$-algebras, {\it J.~Pure Appl. Algebra} {\bf 5}
  (1974), 51-96.
  

\bibitem{STR} \c{S}.~Str\u{a}til\u{a}, {\it Modular theory in operator
  algebras}, Editura Academiei  Bucare\c{s}ti and Abacus Press
  Tunbridge Wells, 1981.

\bibitem{takesaki} M.~Takesaki, {\it Theory of Operator Algebras I},
  Springer-Verlag, 1979.

\bibitem{watatani} Y.~Watatani, Index for $C^*$-subalgebras, {\it Mem. Amer. Math. Soc.} {\bf 83}, 1990.

\end{thebibliography}
\end{document}